\documentclass{amsart}

\usepackage{graphicx,afterpage,float}

\usepackage{amsmath}
\usepackage{amsfonts}
\usepackage{amssymb}
\usepackage{amsthm}
\usepackage{mathrsfs}
\usepackage{wrapfig}
\usepackage{tikz}
\usetikzlibrary{matrix,arrows,decorations.pathmorphing}
\usepackage{url}

\newcommand{\bb}[1]{\mathbb{#1}}
\newcommand{\Z}[1]{\bb{Z}_{#1}}
\newcommand{\F}[1]{\bb{F}_{#1}}

\newcommand{\<}[1]{\langle #1 \rangle}

\DeclareMathOperator{\Aut}{Aut}
\DeclareMathOperator{\Cay}{Cay}
\DeclareMathOperator{\AGL}{AGL}

\DeclareMathOperator{\PGammaL}{P\Gamma L}
\DeclareMathOperator{\EGamma}{E\Gamma}
\DeclareMathOperator{\VGamma}{V\Gamma}
\DeclareMathOperator{\EDelta}{E\Delta}
\DeclareMathOperator{\VDelta}{V\Delta}
\DeclareMathOperator{\ESigma}{E\Sigma}
\DeclareMathOperator{\VSigma}{V\Sigma}
\DeclareMathOperator{\GL}{GL}
\DeclareMathOperator{\PGL}{PGL}
\DeclareMathOperator{\PG}{PG}
\DeclareMathOperator{\SL}{SL}
\DeclareMathOperator{\PSL}{PSL}
\DeclareMathOperator{\Soc}{Soc}
\DeclareMathOperator{\Sym}{Sym}

\DeclareMathOperator{\chr}{char}
\DeclareMathOperator{\Hol}{Hol}
\DeclareMathOperator{\val}{val}
\DeclareMathOperator{\wreath}{wr}
\DeclareMathOperator{\Fix}{Fix}

\newtheorem{theorem}{Theorem}[section]
\newtheorem{corollary}[theorem]{Corollary}

\newtheorem{lemma}[theorem]{Lemma}
\newtheorem{proposition}[theorem]{Proposition}

\theoremstyle{definition}
\newtheorem{definition}[theorem]{Definition}
\newtheorem{construction}{Construction}
\newtheorem{example}[theorem]{Example}
\newtheorem{remark}[theorem]{Remark}
\newtheorem{notation}{Notation}

\title{Normal Edge-Transitive Cayley Graphs of Frobenius Groups}

\author[B.~Corr]{Brian P. Corr}
\address{Brian P. Corr, Centre for Mathematics of Symmetry and Computation,\newline
School of Mathematics and Statistics,\newline
The University of Western Australia,
 Crawley, WA 6009, Australia}
\email{brian.p.corr@gmail.com}

\author[C. E. Praeger]{Cheryl E. Praeger}
\address{Cheryl E. Praeger, Centre for Mathematics of Symmetry and Computation,\newline
School of Mathematics and Statistics,\newline
The University of Western Australia,
 Crawley, WA 6009, Australia\newline
Also affiliated with King Abdulaziz University,
Jeddah, Saudi Arabia} \email{Cheryl.Praeger@uwa.edu.au}

\thanks{The first author is supported by an Australian Mathematical Society Lift-Off Fellowship.}

\tikzstyle{normal}=[circle, draw, fill=black!30,
                        inner sep=0pt, minimum width=6pt]

\tikzstyle{red}=[circle, draw, fill=red,
                        inner sep=0pt, minimum width=6pt]

\tikzstyle{green}=[circle, draw, fill=green,
                        inner sep=0pt, minimum width=6pt]

\tikzstyle{blue}=[circle, draw, fill=blue,
                        inner sep=0pt, minimum width=6pt]

\tikzstyle{yellow}=[circle, draw, fill=yellow,
                        inner sep=0pt, minimum width=6pt]

\tikzstyle{biggraph}=[circle, draw, fill=black!30,
                        inner sep=0pt, minimum width=5pt]
\tikzstyle{node1}=[circle, draw, fill=black!30,
                        inner sep=0pt, minimum width=13pt]

\tikzstyle{design}=[circle, draw, fill=black!30,
                        inner sep=0pt, minimum width=17pt]

\begin{document}
\begin{abstract}
A Cayley Graph for a group $G$ is called normal edge-transitive if it admits an
edge-transitive action of some subgroup of the Holomorph of $G$ (the normaliser
of a regular copy of $G$ in $\Sym(G)$). We complete the classification of normal
edge-transitive Cayley graphs of order a product of two primes by dealing with
Cayley graphs for Frobenius groups of such orders. We determine the automorphism
groups of these graphs, proving in particular that there is a unique
vertex-primitive example, namely the flag graph of the Fano plane.
\end{abstract}
\maketitle
%
%
%
\section{Introduction}\label{section:intro}
Normal edge-transitive Cayley graphs were identified by the second author
\cite{praeger1999finite} in 1999 as a family of central importance for
understanding
Cayley graphs in general. Such graphs have a subgroup of automorphisms which is
transitive on edges and which normalises a copy of the group used to construct
the Cayley graph. Moreover each normal edge-transitive Cayley graph was shown to
have, as a `normal quotient', a normal edge-transitive Cayley graph for a
characteristically simple group. This raised the question of reconstructing
normal edge-transitive Cayley graphs from a given normal quotient. In this paper
we answer the question in the smallest case, where the normal quotient has prime
order $q$ and the group $G$ of interest has order $pq$, where $p$ also is prime.
\\\\
The case of abelian groups $G$ was treated in the MSc thesis of Houlis
\cite{houlis}, and in this paper we complete the nonabelian case. That is to
say, we classify normal edge-transitive Cayley graphs for Frobenius groups of
order a product of two primes, and investigate the automorphism groups of these
graphs.
\begin{theorem}\label{intro}
Let $G_{pq}$ be a Frobenius group of order $pq$, where $p$ and $q$ are primes
and $q<p$, and let $\Gamma$ be a connected normal edge-transitive Cayley graph
for $G_{pq}$ and $Y=\Aut\Gamma$. Then one of the following holds:
\begin{enumerate}
\item $\Gamma\cong C_q[\overline{K_p}]$, with $Y\cong S_p\wreath D_{2q}$ if $q$
is odd and $Y\cong S_p \wreath S_2$ if $q=2$;
\item $\Gamma\cong K_p \times C_q$, with $Y\cong S_p \times D_{2q}$ if $q$ is odd; or $K_p \times K_2$ with $Y = S_p\times \Z{2}$ if $q=2$; or
\item $\Gamma=\Gamma(pq,\ell,i)$ as defined in Construction \ref{construction2},
for some proper divisor $\ell$ of $p-1$ with $\ell >1$, and either $(q,i)=(2,1)$
or $1\leq i \leq (q-1)/2$, and 
\[
Y\cong
\begin{cases} 
G_{pq}.\Z{\ell} & \text{when $q=2$ or $q \nmid \ell$,} \\ 
G_{pq}.\Z{\ell}.\Z{2} & \text{when $q\geq 3$, $q \mid \ell$,}
\end{cases}
\]
or $Y$ is one of the four exceptions listed in Table \ref{introtable}.
\end{enumerate}
\end{theorem}
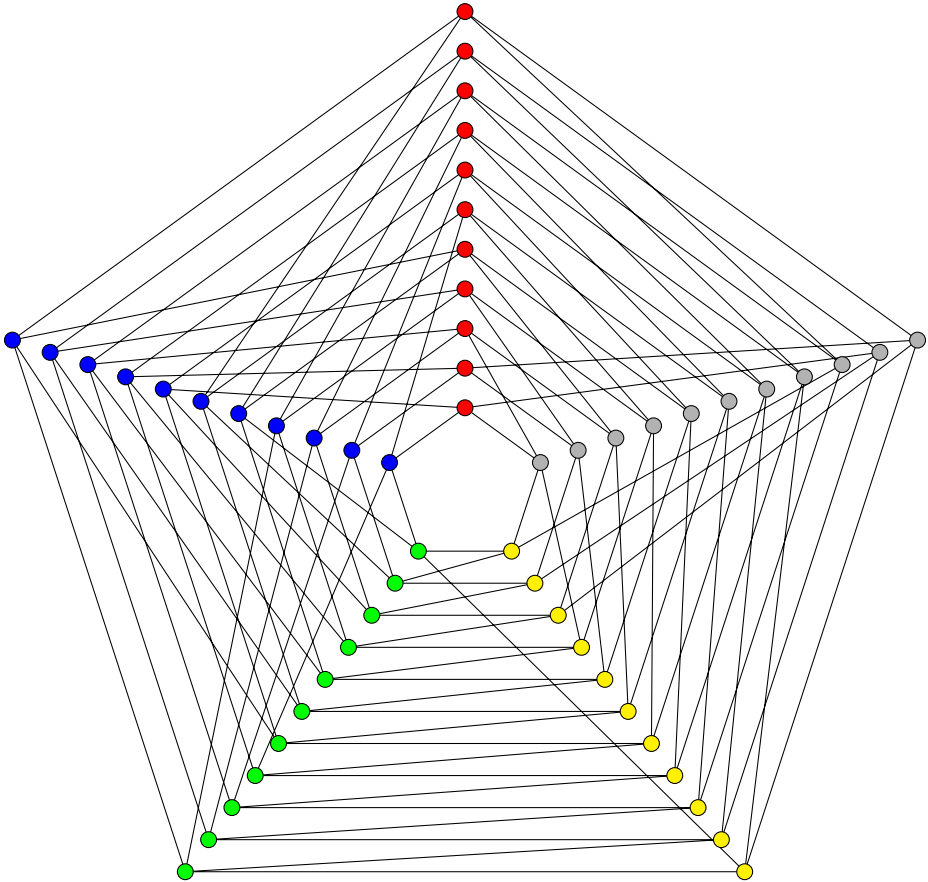
\begin{figure}\label{pic:11521}
\begin{center}
\begin{tikzpicture}[x=15, y=15]
\draw
(18:2) node[normal]{} -- (306:2)
(18:2)                       -- (306:5)

(18:3) node[normal]{} -- (306:3)
(18:3)                       -- (306:6)

(18:4) node[normal]{} -- (306:4)
(18:4)                       -- (306:7)

(18:5) node[normal]{} -- (306:5)
(18:5)                       -- (306:8)

(18:6) node[normal]{} -- (306:6)
(18:6)                       -- (306:9)

(18:7) node[normal]{} -- (306:7)
(18:7)                       -- (306:10)

(18:8) node[normal]{} -- (306:8)
(18:8)                       -- (306:11)

(18:9) node[normal]{} -- (306:9)
(18:9)                       -- (306:12)

(18:10) node[normal]{} -- (306:10)
(18:10)                       -- (306:2)

(18:11) node[normal]{} -- (306:11)
(18:11)                       -- (306:3)

(18:12) node[normal]{} -- (306:12)
(18:12)                       -- (306:4)

(306:2) node[yellow]{} -- (234:2)
(306:2)                       -- (234:3)

(306:3) node[yellow]{} -- (234:3)
(306:3)                       -- (234:4)

(306:4) node[yellow]{} -- (234:4)
(306:4)                       -- (234:5)

(306:5) node[yellow]{} -- (234:5)
(306:5)                       -- (234:6)

(306:6) node[yellow]{} -- (234:6)
(306:6)                       -- (234:7)

(306:7) node[yellow]{} -- (234:7)
(306:7)                       -- (234:8)

(306:8) node[yellow]{} -- (234:8)
(306:8)                       -- (234:9)

(306:9) node[yellow]{} -- (234:9)
(306:9)                       -- (234:10)

(306:10) node[yellow]{} -- (234:10)
(306:10)                       -- (234:11)

(306:11) node[yellow]{} -- (234:11)
(306:11)                       -- (234:12)

(306:12) node[yellow]{} -- (234:12)
(306:12)                       -- (234:2)

(234:2) node[green]{} -- (162:2)
(234:2)                       -- (162:6)

(234:3) node[green]{} -- (162:3)
(234:3)                       -- (162:7)

(234:4) node[green]{} -- (162:4)
(234:4)                       -- (162:8)

(234:5) node[green]{} -- (162:5)
(234:5)                       -- (162:9)

(234:6) node[green]{} -- (162:6)
(234:6)                       -- (162:10)

(234:7) node[green]{} -- (162:7)
(234:7)                       -- (162:11)

(234:8) node[green]{} -- (162:8)
(234:8)                       -- (162:12)

(234:9) node[green]{} -- (162:9)
(234:9)                       -- (162:2)

(234:10) node[green]{} -- (162:10)
(234:10)                       -- (162:3)

(234:11) node[green]{} -- (162:11)
(234:11)                       -- (162:4)

(234:12) node[green]{} -- (162:12)
(234:12)                       -- (162:5)

(162:2) node[blue]{} -- (90:2)
(162:2)                       -- (90:7)

(162:3) node[blue]{} -- (90:3)
(162:3)                       -- (90:8)

(162:4) node[blue]{} -- (90:4)
(162:4)                       -- (90:9)

(162:5) node[blue]{} -- (90:5)
(162:5)                       -- (90:10)

(162:6) node[blue]{} -- (90:6)
(162:6)                       -- (90:11)

(162:7) node[blue]{} -- (90:7)
(162:7)                       -- (90:12)

(162:8) node[blue]{} -- (90:8)
(162:8)                       -- (90:2)

(162:9) node[blue]{} -- (90:9)
(162:9)                       -- (90:3)

(162:10) node[blue]{} -- (90:10)
(162:10)                       -- (90:4)

(162:11) node[blue]{} -- (90:11)
(162:11)                       -- (90:5)

(162:12) node[blue]{} -- (90:12)
(162:12)                       -- (90:6)

(90:2) node[red]{} -- (18:2)
(90:2)                       -- (18:11)

(90:3) node[red]{} -- (18:3)
(90:3)                       -- (18:12)

(90:4) node[red]{} -- (18:4)
(90:4)                       -- (18:2)

(90:5) node[red]{} -- (18:5)
(90:5)                       -- (18:3)

(90:6) node[red]{} -- (18:6)
(90:6)                       -- (18:4)

(90:7) node[red]{} -- (18:7)
(90:7)                       -- (18:5)

(90:8) node[red]{} -- (18:8)
(90:8)                       -- (18:6)

(90:9) node[red]{} -- (18:9)
(90:9)                       -- (18:7)

(90:10) node[red]{} -- (18:10)
(90:10)                       -- (18:8)

(90:11) node[red]{} -- (18:11)
(90:11)                       -- (18:9)

(90:12) node[red]{} -- (18:12)
(90:12)                       -- (18:10)
;
\end{tikzpicture}
\end{center}
\caption{The Graph $\Gamma(55,2,1)$ as in Construction \ref{construction1}. For a second example with the same number of vertices and valency, see Fig. \ref{pic:11522}.}
\end{figure}
\begin{center}
\begin{table}
\begin{tabular}[h]{|l|ccc|}
\hline
$\qquad\Gamma$                   & $\val\Gamma$ & $\Aut\Gamma$ & Reference \\
$\Gamma(21,2,1)$                 & 4            & $\PGL(3,2).\Z{2}$  & Section
\ref{fano}, Figure \ref{primitive picture}, Proposition \ref{primitive} \\
$\Gamma(22,5,1)$                 & 5            & $\PGL(2,11).\Z{2}$ & Section
\ref{Section IncidenceGraphs} \\
$\Gamma(14,3,1)$                 & 3            & $\PGL(3,2).\Z{2}$ & Section
\ref{Section IncidenceGraphs} \\
$\Gamma(146,9,1)$                & 9            & $\PGL(3,8).\Z{2}$ & Section
\ref{Section IncidenceGraphs} \\
\hline
\end{tabular}
\caption{{Exceptional Normal Edge-transitive Cayley graphs for
$G_{pq}$}}\label{introtable}
\end{table}
\end{center}
The restriction to connected graphs is allowable as discussed in Section
\ref{connectivitysection}.
\begin{remark}
\begin{enumerate}
\item There is a unique vertex-primitive, normal edge-transitive Cayley Graph of
a Frobenius group $G_{pq}$ of order $pq$, namely $\Gamma(21,2,1)$ as defined in
Construction \ref{construction2}, and it is isomorphic to the flag graph
$\Gamma_F$ of the Fano Plane (see Section \ref{fano}, Figure \ref{primitive picture} and Proposition
\ref{primitive}).
\item If $q\mid \ell$ then $\Gamma(pq,\ell,i)$ is a Cayley Graph for
$\Z{p}\times \Z{q}$ as well as a Cayley graph for $G_{pq}$ (see Proposition
\ref{qdividesl}).
\item If $q\leq 3$ or $q\mid \val\Gamma$ then $\Gamma$ is the unique connected,
normal edge-transitive Cayley graph of order $pq$ and valency $\val\Gamma$ up to
isomorphism (see Proposition \ref{qdividesl}) if $q \geq 5$ and $q\nmid
\val\Gamma$ then there are exactly
$\frac{q-1}{2}$ such graphs up to isomorphism (Corollary \ref{nonisomorphic}).
\item The exceptional graphs in Table \ref{introtable} are all well known:
$\Gamma(21,2,1)$ is the flag graph of the Fano plane, $\Gamma(22,5,1)$ is the
incidence graph of the $(11,5,2)$-biplane, and $\Gamma(14,3,1)$ and
$\Gamma(146,9,1)$ are the incidence graphs of the Fano plane $\PG(2,2)$ and of
$\PG(2,8)$ respectively (see Section \ref{Section IncidenceGraphs}).
\end{enumerate}
\end{remark}
\begin{corollary}
The normal edge-transitive Cayley graphs of order a product of two primes are
known, and are given in Sections \ref{abelian} and \ref{nonabelian}.
\end{corollary}
Section 2 presents essential results about permutation groups and the structure
of normal edge-transit\-ive Cayley graphs, and outlines the strategy for
classification. In Section 3 we summarise Houlis' classification of normal
edge-transitive Cayley graphs for abelian groups of order a product of two
primes (since his results were not published) and we classify the normal
edge-transitive Cayley graphs for $G_{pq}$. In Section 5 we resolve questions of
redundancy in our classification and determine the full automorphism groups of
the graphs obtained.
\section{Background and Examples}
A subset $S$ of a group $G$ is called a {\it Cayley Subset} if $1_G\notin S$ and
$S$ contains $s^{-1}$ for every $s\in S$. For a Cayley subset $S$, the {\it
Cayley Graph} $\Gamma=\Cay(G,S)$ has vertex set $\VGamma=G$, and edges the pairs
$\{x,y\}$ for which $yx^{-1}\in S$. Each such graph admits the group
$\rho(G)\cong G$, acting by right multiplication $\rho(g):x\mapsto xg$, as a
subgroup of the automorphism group $\Aut\Gamma$, and $\Gamma$ is called
\emph{normal edge-transitive} if $N_{\Aut\Gamma}(\rho(G))$ is transitive on the
edges of $\Gamma$ (see Section \ref{background} or \cite{praeger1999finite}). 
\begin{remark}
Normal edge-transitivity is a property that depends upon the group $G$ \emph{as
well as the graph $\Gamma$}. For example, for any group $G$, we have that
$\Cay(G,G \setminus \{1\}) \cong K_{|G|}$ is always an edge-transitive graph,
but its normal edge-transitivity is not guaranteed.
\end{remark}
The group $\Aut(G)_S$ of (group) automorphisms of $G$ which fix $S$ setwise is
an automorphism group of $\Gamma=\Cay(G,S)$, and
$N_{\Aut\Gamma}(\rho(G))=\rho(G).\Aut(G)_S$ (see, for example,
\cite[pp.6-7]{praeger1999finite}). Hence $\Gamma$ is normal edge-transitive when\\ 
$\rho(G).\Aut(G)_S$ is transitive on the edge-set.\\\\
Given a graph $\Gamma$ and a partition $\mathscr{P}$ of the vertex set
$\VGamma$, the \emph{quotient graph} $\Gamma_{\mathscr{P}}$ has vertex set
$\mathscr{P}$, with two blocks $B,B'$ adjacent if there exists a pair of
adjacent vertices $\alpha, \alpha'\in \VGamma$ with $\alpha\in B$ and
$\alpha'\in B'$. For an edge-transitive subgroup $A=\rho(G).A_0$ of
$\rho(G).\Aut(G)_S$, a \emph{normal quotient} of a Cayley graph
$\Gamma=\Cay(G,S)$ is the quotient $\Gamma_{\mathscr{P}}$, where $\mathscr{P}$
is the set of orbits of an $A_0$-invariant normal subgroup $M$ of $G$ and is
equal to $\Cay(G/M,SM/M)$ (see \cite[Theorem 3]{praeger1999finite}); we denote
this
quotient by $\Gamma_M$. The quotient $\Gamma_M$ admits an (unfaithful) normal
edge-transitive action of $A$, with kernel $\rho(M).C_{A_0}(G/M)$. In particular
each proper characteristic subgroup $M$ of $G$ is $A_0$-invariant, yielding a
nontrivial normal quotient $\Gamma_M$ of $\Gamma$ which is a normal
edge-transitive Cayley graph.\\\\
The `basic' members of the class of finite normal edge-transitive Cayley graphs
were thus identified in \cite{praeger1999finite} as Cayley graphs for
characteristically
simple groups $H$ relative to a subgroup $A_0$ of $\Aut(H)$, leaving invariant
no proper nontrivial normal subgroups of $H$.\\\\
To investigate the basic normal edge-transitive Cayley graphs, a natural
starting point is $G=\Z{q}$, with $q$ a prime; normal edge-transitive Cayley
graphs for these groups were described in \cite[Example 2]{praeger1999finite}.
The result
follows easily from Chao's classification of symmetric (i.e. arc-transitive)
graphs on $q$ vertices \cite{chao1971classification}. The basic graphs are {\it
circulants}
(essentially, edge-unions of cycles). We discuss this case in more detail in
Section \ref{prime}.\\\\
With the simplest case complete, we look for multicovers of these most basic
cases, but again we seek to identify a kind of `basic' reconstruction. Suppose
that $\Gamma=\Cay(G,S)$ is normal edge-transitive relative to $A=\rho(G)A_0$,
where $A_0\leqslant\Aut(G)_S$, and that $\Gamma$ is a normal multicover of
$\Gamma_N$, where $N$ is an $A_0$-invariant normal subgroup of $G$. We say
$\Gamma$ is a {\it minimal normal multicover} of $\Gamma_N$ relative to $A$ if
there is no way to get to $\Gamma_N$ in more than one step from $\Gamma$: that
is, there is no $A_0$-invariant nontrivial normal subgroup of $G$ properly
contained in $N$.\\\\
Again we see that the smallest case is when the index $|G:N|$ is prime, and $G$
has order a product of two primes $p,q$. The groups $G$ to consider are the
Abelian groups $\Z{q^2}$ (with $p=q$) and $\Z{p}\times \Z{q}$, and the
nonabelian {\it Frobenius group} $F_{pq}$, when $p\equiv 1 \pmod{q}$. The
classification in the abelian cases was completed by Houlis in his MSc Thesis
\cite{houlis}. The classification in the final (nonabelian) case is completed in
this paper, and since his thesis remains unpublished we also summarise Houlis'
results (see Section \ref{abelian}).\\\\
Our main result (which we prove in Section \ref{nonabelian}) is the following,
where the graphs $\Gamma(pq)$ and
$\Gamma(pq,\ell,i)$ are defined in Constructions \ref{construction1} and
\ref{construction2}:
\begin{proposition}\label{classification}
Let $\Gamma$ be a connected normal edge-transitive Cayley graph for
$G_{pq}$, where $p,q$ are primes and $q$ divides $p-1$. Let $T$ be the Sylow
$p$-subgroup of $G_{pq}$. Then $\Gamma$ is a normal multicover of $\Gamma_T\cong
K_2$ if $q=2$ or $\Gamma_T\cong C_q$ if $q$ is odd, and $\Gamma$ is one of the
graphs listed in Theorem \ref{intro}.
\end{proposition}
\begin{remark}
The automorphism groups of all connected normal edge-transitive Cayley graphs
for $G_{pq}$ are determined in Proposition \ref{autgamma} below.
\end{remark}
\begin{remark}
If $q\leq 3$, then Construction \ref{construction2} produces a unique graph for
each $\ell$ (by definition of $i$). If $q\geq 5$ and $q\mid \ell$, then the
graphs
$\{\Gamma(pq,\ell,i)\mid 1\leq i \leq (q-1)/2 \}$ are all isomorphic (see
Proposition \ref{qdividesl}). If $q\geq 5$ and $q\nmid\ell$ then they are
pairwise
nonisomorphic (see Corollary \ref{nonisomorphic}).
\end{remark}
\subsection{Normal Edge-transitive Cayley Graphs}\label{background}
Recall that $\rho(G)$ is the subgroup of $\Sym G$ consisting of all permutations
$\rho(g): x\mapsto xg$ for $g\in G$, and that $N:=N_{\Aut\Gamma}(\rho(G))$ is
$\rho(G).\Aut(G)_S$. Note that for normal edge-transitivity $N$ need only be
transitive on \emph{undirected} edges, and may or may not be transitive on arcs
(\emph{ordered} pairs of adjacent vertices). Normal edge-transitivity can be
described group-theoretically as follows. For $g\in G$ and $H\leqslant \Aut G$
we denote by $g^H=\{g^h\mid h\in H\}$ the $H$-orbit of $g$, and we write
$g^{-H}=(g^{-1})^H$.
\begin{lemma}[\cite{praeger1999finite}, Proposition 1(c)]\label{orbits}
Let $\Gamma=\Cay(G,S)$ be an undirected Cayley graph with $S\neq \emptyset$, and
$N=\rho(G).\Aut(G)_S$. Then the following are equivalent:
\begin{enumerate}
\item $\Gamma$ is normal edge-transitive;
\item The set $S=T \cup T^{-1}$, where $T$ is an $\Aut(G)_S$-orbit in $G$;
\item There exists $H\leqslant \Aut(G)$ and $g\in G$ such that $S=g^H \cup
g^{-H}$.
\end{enumerate}
Moreover, $\rho(G).\Aut(G)_S$ is transitive on the arcs of $\Gamma$ if and only
if $\Aut(G)_S$ is transitive on $S$.
\end{lemma}
Hence every normal edge-transitive Cayley graph for a group $G$ is determined by
a (nonidentity) group element $g$ and a subgroup $H$ of $\Aut G$. This motivates
the following definition:
\begin{definition}\label{def1}
For a group $G$, $g\in G\setminus\{1\}$ and $H\leqslant\Aut G$, define
\[\Gamma(G,H,g):=\Cay(G, g^{H} \cup g^{-H}).\]
\end{definition}
\subsubsection{A Comment on Connectivity}\label{connectivitysection}
The connected component of $1$ in $\Cay(G,S)$ is the subgroup $\<{S}$, so
$\Cay(G,S)$ is connected if and only if $S$ generates $G$. The next result
allows us to focus our study on the connected case. Let $m.\Gamma_0$ denote the
union of $m$ vertex-disjoint copies of $\Gamma_0$.
\begin{lemma}\label{connectivity}
Suppose that $\Gamma=\Gamma(G,H,g)$ is normal edge-transitive relative to
$\rho(G).H$, and that $G_0=\<{g^H}$ has index $m>1$ in $G$. Then  $\Gamma\cong
m.\Gamma_0$, where $\Gamma_0=\Gamma(G_0, H_0, g)$ with $H_0=H|_{G_0}$ and
$\Gamma_0$ is normal edge-transitive relative to $\rho(G_0).H_0$. 
\end{lemma}
\begin{proof}
The connected components of $\Gamma=\Cay(G,S)$ (where $S=g^H\cup g^{-H}$) have
as their vertex sets the right cosets of $G_0=\<{S}$, and are transitively
permuted by $\rho(G)$. In particular they are all isomorphic to
$\Gamma_0=\Cay(G_0,S)$, the component containing $1$, and in the notation of
Definition \ref{def1}, $\Gamma_0=\Gamma(G_0,H_0,g)$, where $H_0$ is the group of
automorphisms of $\<{S}$ induced by $H$. The $H_0$-orbits in $G_0$ are precisely
the $H$-orbits in $G_0$, and so $\Gamma_0$ is normal edge-transitive, by
Proposition \ref{orbits}.
\qed\end{proof}
\begin{remark}\label{disconnected}
Since $\Aut(m.\Gamma_0)=\Aut\Gamma_0\wreath S_m$, disconnected normal
edge-transitive Cayley graphs are easily constructed from connected ones. For
example, given a connected normal edge-transitive Cayley graph
$\Gamma_0=\Gamma(G_0,H_0,g)$ and a group $M$ of order $m$, identifying $H_0$
with the subgroup $H=H_0 \times 1$ of $G=G_0\times M$ and $g$ with the element
$(g,1_M)$, the graph $\Gamma(G,H,g)\cong m.\Gamma_0$, and $\rho(G).H$ acts
edge-transitively.
\end{remark}
Thus we consider only connected graphs, and so henceforth we assume
$\Gamma=\Gamma(G,H,g)$, where $H\leqslant \Aut G$ and $\<{g^H} = G$.
\subsubsection{Use of Symmetry in the Analysis}
A classification of normal edge-transitive Cayley graphs for a given group $G$
is reduced to the study of the action of subgroups of $\Aut G$ on $G$. We employ
this strategy in Section 4. For efficiency we use the following result to avoid
producing too many copies of each example.
\begin{lemma}\label{conjugate}
Let $\sigma\in\Aut G$. Then $\sigma$ induces an automorphism from
$\Gamma(G,H,g)$ to $\Gamma(G,H^{\sigma},g^{\sigma})$. In particular if 
$\sigma\in N_{\Aut G}(H)$, then $\Gamma(G,H,g)\cong \Gamma(G,H,g^{\sigma})$.
\end{lemma}
\begin{proof}
For any $x, y\in G$ we have $xy^{-1} \in S$ if and only if
$x^{\sigma}(y^{\sigma})^{-1}=(xy^{-1})^{\sigma}\in S^{\sigma}$, and so
$\{x,y\}\in\EGamma$ if and only if $\{x^{\sigma},y^{\sigma}\}\in\EGamma'$. Thus
$\sigma$ induces an isomorphism $\Cay(G,S)\to\Cay(G,S^{\sigma})$, which implies
$\Gamma(G,H,g)\cong\Gamma(G,H^{\sigma},g^{\sigma})$. If additionally
$\sigma\in\Aut(G)$ normalises $H$, then $H=H^{\sigma}$ and the result
follows.\qed\end{proof}
In particular for a given subgroup $H$, two elements of the same orbit in the
action of $N_{\Aut G}(H)$ on $H$-orbits in $G$ generate isomorphic graphs, and
so we need only consider a single representative $H$-orbit from each $N_{\Aut
G}(H)$-orbit.
\subsection{Examples \& Constructions}
Here we present several important examples of normal edge-transitive Cayley
graphs. First we describe how Lemma \ref{conjugate} can be used to classify all
normal edge-transitive Cayley graphs of prime order $p$.
\begin{example}\label{prime}
Let $G$ be the additive group of the ring $\Z{p}$ of integers modulo a prime
$p$, and $\Aut(G)=\Z{p}^{*}$, the multiplicative group of units. For every even
divisor $\ell$ of $p-1$, and for $\ell=1$ if $p=2$, there is a unique subgroup
of $\Aut(G)$ of order $\ell$, namely $H_{\ell}:= \<{m^{(p-1)/\ell}}$. The graph
$\Gamma(p,\ell):=\Gamma(G, H_{\ell},1)$ is normal edge-transitive of valency
$\ell$ and since $\Aut(G)$ normalises $H_{\ell}$ and is transitive on the
$H_{\ell}$-orbits in $G\setminus\{0\}$, it follows from Lemma \ref{conjugate}
that every normal edge-transitive Cayley graph for $G$ is isomorphic to
$\Gamma(p,\ell)$ for some $\ell$.
\end{example}
The notion of a product of two graphs may be defined in several ways: we
present two here, each of which arises in our study (see Construction
\ref{construction1} and Lemma \ref{p-1}).
\begin{definition}\label{directproduct}
Given graphs $\Sigma, \Delta$, the \emph{direct product}
$\Gamma=\Sigma\times\Delta$ has vertex set $\VSigma \times \VDelta$, with
vertices $(\alpha_1, \beta_1), (\alpha_2, \beta_2)$ adjacent if both $\alpha_1$
is adjacent to $\alpha_2$ and $\beta_1$ is adjacent to $\beta_2$.
\end{definition}
The direct product $\Sigma\times\Delta$ is so named because the direct product
$\Aut\Sigma \times\Aut\Delta$ is contained in $\Aut(\Sigma\times\Delta)$.
\begin{definition}\label{lex}
Given graphs $\Sigma, \Delta$, the \emph{lexicographic product}
$\Gamma=\Sigma[\Delta]$ has vertex set $\VSigma \times \VDelta$, with vertices
$(\alpha_1, \beta_1), (\alpha_2, \beta_2)$ adjacent if either $\{\alpha_1 ,
\alpha_2\}\in \ESigma$, or both $\alpha_1=\alpha_2$ and $\{\beta_1 ,
\beta_2\}\in \EDelta$.
\end{definition}
If both $\Sigma$ and $\Delta$ are regular, then their lexicographic product is
regular with valency $\val\Delta + |\VDelta| \val\Sigma$. The lexicographic
product has $(\Aut\Delta)\wreath(\Aut\Sigma)$ as a subgroup of automorphisms
(which may be a proper subgroup: for example if $\Delta=\Sigma=K_2$ then
$\Gamma=\Sigma[\Delta]=K_4$ and $(\Aut\Delta)\wreath(\Aut\Sigma)=D_8 < S_4= \Aut
K_4$).\\\\
The following result determines a sufficient condition for a Cayley graph to
have a decomposition as a lexicographic product. If $\Gamma=\Cay(G,S)$ and $M
\unlhd G$, then the normal quotient $\Gamma_M$ of
$\Gamma$ is $\Cay(G/M, SM/M)$ (see \cite[Theorem 3(b)]{praeger1999finite}). For
a graph
$\Gamma$ and vertex $\alpha$ we
denote by $\Gamma(\alpha)$ the set of vertices adjacent to $\alpha$ in $\Gamma$.
Note that, in $\Gamma=\Gamma(G,H,g)$ we have $\Gamma(g)=Sg$, where $S=g^H\cup
g^{-H}$.
\begin{proposition}\label{cosets}
Let $G$ be a group, $M$ a normal subgroup with $m=|M|$, and let
$\Gamma=\Cay(G,S)$ be a connected Cayley graph for $G$. Then $\Gamma\cong
\Gamma_M[\overline{K_{m}}]$ if and only if $S$ is a union of cosets of $M$.
\end{proposition}
\begin{proof}
Since $\Gamma_M=\Cay(G/M,SM/M)$, by Definition \ref{lex} it follows that
$\Gamma\cong \Gamma_M[\overline{K_{m}}]$ if and only if $\Gamma(1_G)=S$ is equal
to $SM$, that is to say, $S$ is a union of $M$-cosets.
\qed\end{proof}
\begin{remark}
If the graph $\Gamma$ in Proposition \ref{cosets} is normal edge-transitive
relative to $N\leqslant N_{\Aut\Gamma}(\rho(G))$, and if $N$ normalises $M$,
then $\Gamma_M$ is also normal edge-transitive, as it is a normal quotient of
$\Gamma$ (see \cite[Theorem 3]{praeger1999finite}). 
\end{remark}
\subsection{Permutation Groups and Group Actions}\label{permutations}
We use the basic definitions and notation found in \cite{dixonmortimer}, and we
assume $G$ is a finite group acting on a set $\Omega$. We denote by $\rho,
\lambda,\iota$ respectively the right, left and conjugation actions of $G$ on
itself.\\\\
A key concept in our analysis is {\it primitivity}. Let $G$ be a group acting
transitively on $\Omega$. A subset $\Delta$ of $\Omega$ is called a {\it block
of imprimitivity} if, for every $g\in G$, the image $\Delta^g$ is either empty
or equal to $\Delta$. For any block $\Delta$, the set $\{\Delta^g \mid g\in G\}$
is a $G$-invariant partition of $\Omega$, or {\it system of imprimitivity} for
$G$ in $\Omega$. The singleton sets and the whole set $\Omega$ are always
blocks; for this reason they are called {\it trivial} blocks. A group $G$ acts
{\it primitively} on $\Omega$ if the only blocks are the trivial blocks. In
particular if $B$ is a block and $\alpha\in B$, then $G_{\alpha}$ fixes $B$
setwise. Consequently if $G_{\alpha}$ is transitive on
$\Omega\setminus\{\alpha\}$, that is, if $G$ is {\it 2-transitive}, then
$B=\{\alpha\}$ or $B=\Omega$, and so $G$ is primitive.
\begin{lemma}\label{affineortwotransitive}
Let $G$ be a transitive permutation group of prime degree $p$. Then $G$ is
primitive, and either $G\leqslant \AGL(1,p)$, or $G$ is almost simple and
$2$-transitive with socle $T$, where $p, T$ and the Schur multiplier of $T$ are
as in one of the lines of Table \ref{as2t}.
\end{lemma}
The socle $T$ of an almost simple group $G$ is its unique minimal normal
subgroup (which is a nonabelian simple group). The possibilities for $p$ and $T$
can be obtained from, for example, Cameron \cite[Table 7.4]{cameron}. Their
classification depends on the finite simple group classification. The Schur
multiplier $M(T)$ is obtained from \cite[Section 8.4]{karpilovsky1987schur}.\\\\
\begin{table}[t]
\begin{center}
\begin{tabular}{|l|c|l|l|}
\hline
$T=\Soc(G)$ & $p$ & $M(T)$ & Conditions \\
\hline
$A_p$       & $p$  & $\Z{2}$, except & $p\geq 5$ \\
& &  $\Z{6}$ if $p=7$ &\\
$\PSL(n,r)$ & $\frac{r^n-1}{r-1}$ & $\Z{gcd(n,r-1)}$, except & $n \geq 2$ \\
& &  $\Z{2}$ if $(n,r)=(2,4),$ & $(n,r)\neq (2,2)$ \\
& & $(3,2), (3,3)$&  $n$ prime \\
$\PSL(2,11)$ & 11 & $\Z{2}$ & \\
$M_{11}$ & 11 & 1 & \\
$M_{23}$ & 23 & 1 & \\
\hline
\end{tabular}
\caption{Simple Normal Subgroups $T$ and Schur Multipliers $M(T)$
of almost simple 2-transitive groups of prime degree $p$.}\label{as2t}
\end{center}
\end{table}
The next result gives several ways in which blocks of imprimitivity may arise:
\begin{lemma}\label{blocks}
Let $G$ be a transitive permutation group on $\Omega$, and let $N\triangleleft
G$. Then:
\begin{enumerate} 
\item the $N$-orbits in $\Omega$ are blocks of imprimitivity for $G$. In
particular if $G$ is primitive then $N$ is transitive on $\Omega$, or is
trivial; and 
\item For $\alpha\in\Omega$, the set $\Fix N_{\alpha}=\{\beta\in\Omega \mid
\beta^n=\beta \quad \forall  n\in N_{\alpha}\}$ is a block of imprimitivity for
$G$.
\end{enumerate}
\end{lemma}
\begin{proof}
(i) See \cite[p.13]{wielandt1964finite}. (ii) For any $g\in G$ we have $(\Fix
N_{\alpha})^g=\Fix N_{\alpha^g}$. If $\beta\in \Fix N_{\alpha}$  then since $G$
is transitive and finite, and since $N$ is normal, we have
$N_{\alpha}=N_{\beta}$, so $\Fix N_{\alpha}=\Fix N_{\beta}$.\\\\
Now let $\Delta=\Fix N_{\alpha}$, and suppose $\gamma\in\Delta^g \cap \Delta$.
The $\gamma=\delta^g$ with $\gamma,\delta\in\Delta$, and so by the previous
paragraph $\Delta = \Fix N_{\delta}=\Fix N_{\gamma}$. But $\Delta^g=\Fix
N_{\delta}^g=\Fix N_{\delta^g} = \Fix N_{\gamma}=\Delta$, so $\Delta$ is a block
of imprimitivity for $G$.
\qed\end{proof}
Our analysis in Section \ref{automorphism} deals, for the most part, with
imprimitive groups. Given a transitive group $G$ and a nontrivial system of
imprimitivity $\mathscr{B}$, several permutation groups of smaller degree
present themselves naturally. First, the group $G$ acts transitively on the set
of blocks, inducing a subgroup $G^{\mathscr{B}}$ of $\Sym\mathscr{B}$. The
stabiliser $G_B$ of a block $B\in\mathscr{B}$ in this action induces a
transitive subgroup $G_B^B$ of $\Sym B$. These two actions play an important
role in the structure of $G$; in particular for distinct blocks
$B,B'\in\mathscr{B}$ the induced groups $G_B^B$ and $G_{B'}^{B'}$ are
permutationally isomorphic.\\\\
The kernel $K=G_{(\mathscr{B})}$ of the $G$-action on $\mathscr{B}$ acts on each
block $B\in \mathscr{B}$. We say that the $K$-actions on $B$ and $B'$ are {\it
equivalent} if there exists a bijection $\varphi:B\to B'$ such that for every
$\alpha\in B, k\in K$, we have $(\alpha^{\varphi})^k=(\alpha^k)^{\varphi}$. 
\begin{lemma}\label{actionofK}
Let $\mathscr{B}=\{ B_1, B_2, \ldots, B_k\}$. Then the set \[\Sigma = \{ B_i
\mid K^{B_i}\text{ is equivalent to }K^{B_1}\}\] is a block of imprimitivity for
the action of $G$ on $\mathscr{B}$.
\end{lemma}
\begin{proof}
Write $B_i \sim B_j$ when the $K$-actions on $B_i,B_j$ are equivalent. Then
$\sim$ is an equivalence relation on $\mathscr{B}$, and $\Sigma$ is an
equivalence class. Moreover if $B_i\sim B_j$ with bijection $\varphi:B_i\to
B_j$ and $g\in G$ then it is easy to prove $g^{-1}\circ\varphi \circ g: B_i^g\to
B_j^g$ defines an equivalence of the $K$-actions on these blocks, so $B_i^g\sim
B_j^g$. Thus $\sim$ determines a $G$-invariant partition and so $\Sigma$ is a
block of imprimitivity.
\qed\end{proof}
\begin{corollary}\label{actionofK2}
Suppose $G$ acts primitively on the set $\mathscr{B}$ of blocks. Then the
$K$-actions on the blocks in $\mathscr{B}$ are either pairwise equivalent or
pairwise
nonequivalent.
\end{corollary}
\subsection{The Flag Graph of the Fano Plane}\label{fano}
Here we construct the flag-graph of the Fano Plane, which coincidentally appears
as the unique vertex-primitive normal edge-transitive Cayley graph.\\\\
The {\it Fano Plane} $F$ is the smallest nontrivial projective plane; it has
seven points and seven lines, namely the 1-spaces and 2-spaces of the vector
space $\F{2}^3$ with incidence given by inclusion. The automorphism group of $F$
is the simple group $\PSL(3,2)$ of order 168. A {\it flag} of $F$ is an
incident point-line pair. Define the Flag Graph $\Gamma_F$ of the Fano Plane as
follows: the vertices are the 21 flags of $F$, with two flags $(P,L)$ and $(P',
L')$ incident if either $P=P'$ or $L=L'$ (see Figure \ref{primitive picture}). Any automorphism of the Fano Plane
induces an automorphism of $\Gamma_F$, and there is also an extra isomorphism
induced by a duality -- that is, a map that swaps points and lines and preserves
incidence. Thus $\Aut\Gamma_F=\PGL(3,2).\Z{2}\cong \PGL(2,7)$, and
$\Aut\Gamma_F$ acts primitively on $\Gamma_F$.
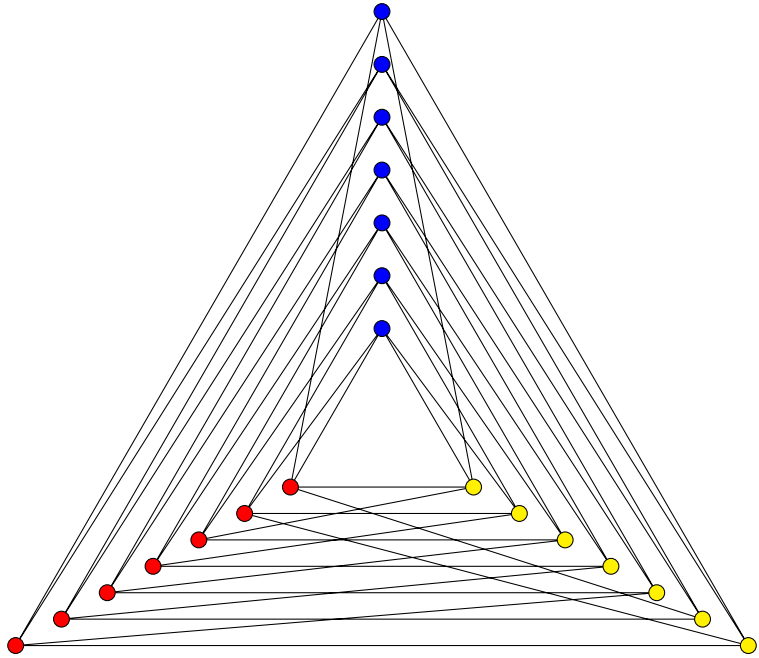
\begin{figure}[t]
\begin{center}
\begin{tikzpicture}[x=20, y=20]
\draw
(90:2) node[blue]{} -- (330:2)
(90:2)                       -- (330:3)

(90:3) node[blue]{} -- (330:3)
(90:3)                       -- (330:4)

(90:4) node[blue]{} -- (330:4)
(90:4)                       -- (330:5)

(90:5) node[blue]{} -- (330:5)
(90:5)                       -- (330:6)

(90:6) node[blue]{} -- (330:6)
(90:6)                       -- (330:7)

(90:7) node[blue]{} -- (330:7)
(90:7)                       -- (330:8)

(90:8) node[blue]{} -- (330:8)
(90:8)                       -- (330:2)

(330:2) node[yellow]{} -- (210:2)
(330:2)                       -- (210:4)

(330:3) node[yellow]{} -- (210:3)
(330:3)                       -- (210:5)

(330:4) node[yellow]{} -- (210:4)
(330:4)                       -- (210:6)

(330:5) node[yellow]{} -- (210:5)
(330:5)                       -- (210:7)

(330:6) node[yellow]{} -- (210:6)
(330:6)                       -- (210:8)

(330:7) node[yellow]{} -- (210:7)
(330:7)                       -- (210:2)

(330:8) node[yellow]{} -- (210:8)
(330:8)                       -- (210:3)

(210:2) node[red]{} -- (90:2)
(210:2)                       -- (90:8)

(210:3) node[red]{} -- (90:3)
(210:3)                       -- (90:2)

(210:4) node[red]{} -- (90:4)
(210:4)                       -- (90:3)

(210:5) node[red]{} -- (90:5)
(210:5)                       -- (90:4)

(210:6) node[red]{} -- (90:6)
(210:6)                       -- (90:5)

(210:7) node[red]{} -- (90:7)
(210:7)                       -- (90:6)

(210:8) node[red]{} -- (90:8)
(210:8)                       -- (90:7);
\end{tikzpicture}
\end{center}
\caption{The flag graph of the Fano plane ($\Gamma(7.3, 2,1)$ in the language of Construction \ref{construction2}): the only vertex-primitive graph in the classification.}\label{primitive picture}
\end{figure}

\subsection{Incidence Graphs}\label{Section IncidenceGraphs}
Another construction which appears by coincidence is the \emph{incidence
graphs}. Given a point-line incidence structure $\mathscr{I} =
(\mathscr{P},\mathscr{L})$ (for example, a projective plane), the
incidence graph $\Gamma_{\mathscr{I}}$ has vertex set
$\mathscr{P} \cup \mathscr{L}$, and an edge between $P\in \mathscr{P},
L\in\mathscr{L}$ if and only if $P$ and $L$ are incident in $\mathscr{I}$. That
is, the edges of $\Gamma_{\mathscr{I}}$ are the flags of $\mathscr{I}$. There
are no edges within $\mathscr{P}$ or within $\mathscr{L}$, and so
$\Gamma_{\mathscr{I}}$ is bipartite.\\\\
The automorphism group of $\Gamma_{\mathscr{I}}$ is either isomorphic to $\Aut
\mathscr{I}$, acting naturally on both points and lines, or isomorphic to
$\<{\Aut\mathscr{I} , \tau}$, where $\tau$ is a map switching $\mathscr{P}$ with
$\mathscr{L}$ and preserving incidence. For example, if $\mathscr{I}$ is the
Fano
plane, then $\Gamma_{\mathscr{I}}$ has 14 vertices and valency $3$ (for a total
of 21 edges). This graph has automorphism group $\PGL(2,7)$ -- isomorphic to the
automorphism group of the Flag graph described in Section \ref{fano}, but acting
imprimitively.
This graph $\Gamma_{\mathscr{I}}$ is isomorphic to $\Gamma(14,3,1)$ as defined
in Construction \ref{construction2}.
\section{The Classification}
\subsection{The Abelian Case}\label{abelian}
In \cite{houlis}, Houlis classified the normal edge-transitive Cayley graphs for the groups $\Z{p^2}, \Z{p}\times \Z{p}$ and $\Z{p}\times\Z{q}$, for primes $p,q$. Combining his results with ours completes the classification of normal edge-transitive Cayley graphs for all groups of order a product of two primes.\\\\
Recall from Proposition \ref{orbits} and Definition \ref{def1} that every normal
edge-transitive Cayley graph for a group $G$ is equal to $\Gamma(G,H,g)=\Cay(G,
g^H\cup (g^{-1})^{H})$, where $H\leqslant \Aut(G)$ and $g\in G$. Note that when
$G$ is abelian, the inversion operation is a group automorphism and induces a
graph automorphism, and if $H$ contains it, then we have $\Gamma(G,H,g)=\Cay(G,
g^H)$: if not, then we may replace $H$ with $\<{H, -1}$ without changing
$\Gamma$, and so we may assume without loss of generality that $H$ contains the
inversion operation. We present Houlis' classification of the abelian case here
by giving a representative $H$ and $g$ for each isomorphism class of graphs.
\\\\
Recall from Example \ref{prime} that, for a divisor $\ell$ of $p-1$, $H_{\ell}$
is the unique subgroup of $\Z{p}^*$ of order $\ell$. Let $x$ be a primitive
element of $\Z{p}$ and let $a=(p-1)/\ell$, so $H_{\ell}=\<{x^a}$. Note that
$H_{\ell}$ contains the inversion operation if and only if $\ell$ is even
(unless $p=2$, in which case inversion is trivial).
\subsubsection{The Case $G=\Z{q}\times \Z{q}$ ($p\neq q$)}\label{ZpZq}
Here we characterise all subgroups of $\Aut(\Z{p}\times \Z{q}) = \Z{p}^* \times
\Z{q}^*$ which give rise to connected normal edge-transitive Cayley graphs, and
hence classify such graphs.
\begin{definition}\label{AbelianCaseConditionsDef}
Let $x, y$ be primitive elements of $\Z{p}^*, \Z{q}^*$ respectively, and suppose
that $d_2,d_1,d$ are integers satisfying the following conditions:
\begin{equation}\label{ConditionsEqn}
 d_2 > 0, \quad d_2 \mid (q-1), \quad  d_1 \mid (p-1), \quad 0\leq d < d_1,
\quad d_1d_2 \mid d(q-1)
\end{equation}
We define a subgroup of $\Z{p}^* \times \Z{q}^* = \<{x} \times \<{y}$ as
follows:
\[
 H(d_2, d_1, d) := \<{(x^d, y^{d_2}),(x^{d_1},1) }.
\]
\end{definition}
\begin{theorem}[\cite{houlis}, Theorem 8.1.6]\label{zpzqthm}
Let $p,q$ be primes with $p\neq q$, let $G = \Z{p}\times \Z{q}$, and suppose
that
$\Gamma$ is a connected normal edge-transitive Cayley graph for $G$. Then there
exist integers $d_2,d_1,d$ satisfying the conditions
\eqref{ConditionsEqn}, with $\frac{q-1}{d_2}$ even if $q >2$ and
$\frac{p-1}{\gcd(d,d_1)}$ even if $p > 2$, such that $\Gamma \cong \Gamma(G,
H(d_2,d_1,d), (1,1))$. Moreover $\Gamma$ has valency
$\frac{p-1}{d_1}\frac{q-1}{d_2}$.
\end{theorem}
\begin{remark}\label{Remarkalternatived1d2d}
It is not difficult to see that each subgroup of $\Z{p}^* \times \Z{q}^*$ is
equal to $H(d_2,d_1,d)$ for some $d_2,d_1,d$ satisfying \eqref{ConditionsEqn},
see for example \cite[Section 2.6]{houlis}. However, while every subgroup $H$ of
$\Z{p}^*\times \Z{q}^*$ yields a unique set of
parameters $d_2,d_1,d$, this is not the only way of parametrising $H$: suppose
that $H = H(d_1,d_2,d)$. If $d=0$, set $c_1:=d_2, c_2:=d_1$ and $c:=0$. If $d>
0$ then set
\[
 c_2 := \gcd(d,d_1), \quad c_1 := \frac{d_1d_2}{\gcd(d,d_1)}, \quad c:=
\frac{c_1}{\gcd(c_1, \frac{p-1}{c_2})}.
\]
Then the parameters $c_2,c_1,c$ satisfy the conditions \eqref{ConditionsEqn}
with $p$ and $q$ interchanged, and we hav $H= \<{(x^{c_2},y^c),(1,y^{c_1})}$.
This yields another
parametrisation of $H$ (and hence of the normal edge-transitive Cayley graphs
for $G$). Up to replacing $(d_2,d_1,d)$ by $(c_2,c_1,c)$ the graphs
$\Gamma(G,H,(1,1))$ in Theorem \ref{zpzqthm} are pairwise nonisomorphic (see
\cite[Theorem 8.1.6(III)]{houlis}).
\end{remark}
\subsubsection{The Case $G=\Z{p}\times\Z{p}$}\label{zpzp}
When $p=q$, the automorphism group of $G$ is larger than $\Z{p}^* \times
\Z{q}^*$: there may be an
`interaction' between the two components. Since $G$ is a $2$-dimensional
$\Z{p}$-vector space, we have $\Aut(G)=\GL(2,\Z{p})$. There are two classes of
subgroups $H\leqslant \Aut(G)$ to consider.\\\\
For the first case, choose a divisor $\ell$ of $p-1$ which is even if $p>2$.
Recall that $H_{\ell}$ is the subgroup of $\Z{p}^* = \<{x}$ generated by
$x^{(p-1)/\ell}$ (having order $\ell$). Subgroups $H$ in the first case have
order $p\ell$, for such an $\ell$, and are conjugate to
\[
H:=\left\{ \left(\begin{array}{cc} b & 0 \\ c & d
\end{array}\right) \mid b\neq 0, d\in H_{\ell} \right\}\leqslant \GL(2,p).
\] 
In this case, the graph $\Gamma(G,H, (1,1))$ is the lexicographic product
$\Gamma(\Z{p},H_{\ell},1)[\overline{K_p}]$ (see \cite[Definition 6.1.1(I),
Theorem 6.1.5]{houlis}).\\\\
In the second case, $H$ is a subgroup of the diagonal matrices, and hence is
isomorphic to $H = H(d_2,d_1,d)$ for some parameters $d_2,d_1,d$ satisfying the
conditions
\eqref{ConditionsEqn}. 
\begin{theorem}[\cite{houlis}, Theorem 6.1.5]
Let $p$ be a prime, let $G = \Z{p}\times \Z{q}$, and suppose that
$\Gamma$ is a connected normal edge-transitive Cayley graph for $G$. Then one of
the following holds:
\begin{enumerate}
 \item $\Gamma \cong \Gamma(\Z{p},H_{\ell},1)[\overline{K_p}]$, for $H_{\ell}$
as in Example \ref{prime}, of valency $p\ell$, for some $\ell \mid (p-1)$, with
$\ell$ even if $p >2$; or
 \item $p$ is odd and there exist integers $d_2,d_1,d$ satisfying the conditions
\eqref{ConditionsEqn}, with $\frac{p-1}{d_2}$ and $\frac{p-1}{\gcd(d,d_1)}$
even, such that $\Gamma \cong \Gamma(G, H(d_2,d_1,d), (1,1))$, of valency
$\frac{(p-1)^2}{d_1d_2}$.
\end{enumerate}
\end{theorem}
\subsubsection{The Case $G=\Z{p^2}$}
In this case, we again have two cases:
\begin{theorem}[\cite{houlis} [Theorem 7.1.3]]
Let $p$ be a prime, let $G = \Z{p^2}$, and suppose that $\Gamma$ is a connected,
normal edge-transitive Cayley graph for $G$. Then there exists a divisor $\ell$
of $p-1$, with $\ell$ even if $p>2$, such that:
\begin{enumerate}
 \item $\Gamma \cong \Gamma(\Z{p}, H_{\ell},1)[\overline{K_p}]$, of valency
$p\ell$; or
\item $p$ is odd and $\Gamma \cong \Cay(G,S)$ of valency $\ell$, where $S$ is
the unique subgroup of $\Z{p^2}^*$ of order $\ell$.
\end{enumerate}
\end{theorem}
\subsection{The Frobenius Group of order $pq$}\label{affine}
A nonabelian group $G$ of order $pq$, for primes $p$ and $q$ with $p > q \geq
2$, exists if and only if $q$ divides $p-1$, and is a Frobenius group and unique
up to isomorphism (see for example \cite[Theorem 7.4.11]{greenbook}). In this
section we
construct such a group $G$ as a subgroup of the 1-dimensional affine group
$\AGL(1,p)$, and describe $\Aut G$.\\\\
The \emph{affine group} $A:=\AGL(1, p)$ consists of all affine transformations
$x \mapsto xa + b$ of the field $\Z{p}$ for $a, b \in \Z{p}$, with $a\neq 0$. It
is generated by
\[
t:x \mapsto x+1, \quad m:x \mapsto xm,
\]
where $m$ is a fixed primitive element of $\Z{p}$. The element $t$ has order
$|t|=p$, and $m$ has order $|m|=p-1$. The group $A=\<{m,t}$ is the semidirect
product $\<{t} \rtimes \<{m}$.\\\\
We use $m$ to denote both the primitive element and the transformation induced
by right multiplication by $m$: with this abuse of notation we have that
$m^{-1}tm=t^m$, where the left hand side denotes composition of maps
(i.e. multiplication in the group $A$), and the right hand side denotes the
$m$th
power of the generator $t$. Each element of $A$ may be uniquely expressed as
$m^i t^j$, with $0\leq i \leq p-2$ and $0 \leq j \leq p-1$, and for $k \geq 0$
we have
\begin{equation}\label{formula} (m^i t^j)^{t^k}=m^i t^{j+k(1-m^i)}, \quad(m^i
t^j)^m = m^i t^{jm}. \end{equation}
For a prime $q$ dividing $p-1$ there is a unique subgroup $G_{pq}$ of
$\AGL(1,p)$ of order $pq$; namely $G_{pq} := \<{z,t}$, where $z = m^{(p-1)/q}$.
Since $t^z: x\mapsto x+z$ and $z\neq 1$, it follows that $t^z\neq t$ and hence
that $G_{pq}$ is not abelian. We identify the nonabelian group $G$ of order $pq$
with this subgroup $G_{pq}$, and denote the translation subgroup $\<{t}$ by $T$.
Note that 
\begin{equation}\label{tconj}
z^{-1}tz=t^{m^{(p-1)/q}}.
\end{equation}
In view of the role played by $\Aut G$ in our strategy for classifying normal
edge-transitive Cayley graphs (see \ref{background}), we need to understand the
automorphism group of $G_{pq}$ and its actions. Since $G_{pq}$ is the unique
subgroup of $A$ of order $pq$, it is a characteristic subgroup of $A$. Thus
$G_{pq}$ is invariant under automorphisms of $A$ and in particular under
conjugation by elements of $A$. We denote by $\iota$ the conjugation action
$A\to \Aut G_{pq}$, and with this notation $\iota(A)\leqslant \Aut G_{pq}$. In
fact, equality holds:

\begin{lemma}\label{autg}
Every automorphism of $G=G_{pq}$ is induced by conjugation by an element of
$A=\AGL(1,p)$, that is, $\Aut G = \iota(A) \cong A$.
\end{lemma}
\begin{proof}
It follows from \eqref{formula} that $\ker\iota = C_{A}(G)$ is trivial, and so
$\iota(A)\cong A$. The subgroup $T$ of translations is the unique Sylow
$p$-subgroup of $G:=G_{pq}$, and so is invariant under $\Aut G$. Thus there is
an induced homomorphism $\varphi:\Aut G\to\Aut T$ which is onto since
$\<{\varphi(\iota(m))}\cong \Aut T$, as both are cyclic of order $p-1$. An
automorphism $\sigma \in \ker\varphi$ is uniquely determined by the image
$z^{\sigma}$ of $z$. As $\<{t}$ is normal in $G$, $ztz^{-1}\in T$ and hence is
fixed by $\sigma$. Thus $(ztz^{-1})^{\sigma} = ztz^{-1}$. \\\\
Now $z^{\sigma}=z^x t^y$ for some $x,y$ with $0\leq x\leq q-2, 0\leq y\leq p-1$.
It follows that $(z^x t^y)t(z^x t^y)^{-1} = ztz^{-1}$ and so
$t^{m^{x(p-1)/q}}=t^{m^{(p-1)/q}}$. Hence $m^{(x-1)(p-1)/q}\equiv 1 \pmod{p}$,
or equivalently,  $x\equiv 1 \pmod{q}$, as $m$ is a primitive element of
$\Z{p}$. Since $0\leq x\leq q-2$ it follows that $x=1$ and $z^{\sigma}=z t^y$.
This leaves at most $p$ choices for $z^{\sigma}$ and so $|\ker\varphi|\leq p$,
and $|\Aut G|=(p-1)|\ker\varphi|\leq p(p-1)=|\iota(A)|$. On the other hand
$|\Aut G| \geq |\iota(A)|=p(p-1)$, and it follows that $\Aut G = \iota(A)$.
\qed\end{proof}
Recall from Lemma \ref{connectivity} that a normal edge-transitive Cayley graph
$\Gamma(G,H,g)$ is connected if and only if $g^H$ generates $G$. In particular
if $G$ contains a proper characteristic subgroup which intersects $g^H$
nontrivially, then $g^H$ lies entirely in this subgroup, and $\Gamma$ is not
connected (by Lemma \ref{connectivity}). In the case of $G=G_{pq}$ this implies
that the element $g$ may not have order $p$, since all elements of order $p$ lie
in the characteristic subgroup $T$. It follows then that $o(g)=q$ and that the
unique $\iota(H)$-invariant normal subgroup of $G$ is $T$.\\\\
We investigate the subgroups of $\Aut G\cong \AGL(1,p)$ with a view to applying
the strategy described in Section \ref{background}. Since $T$ has prime order, a
subgroup of $\AGL(1,p)$ either contains $T$ or intersects it trivially. In the
latter case $H$ is cyclic, and we define, for $\ell\mid(p-1)$ and $0\leq j\leq
p-1$,
\[
H_{(\ell,j)}:= \<{m^{(p-1)/\ell}t^j}.
\]
Note that every element of $\AGL(1,p)\setminus T$ has order dividing $p-1$ and
$|H_{(\ell,j)}|=\ell$ for each $j$.\\\\
Under the natural action of $\AGL(1,p)$ on $G_{pq}$, the orbits are $\{1\},
T\setminus\{1\}$ and the left cosets $\{z^i T\mid 1\leqslant i \leqslant q-1\}$.
The induced action of certain subgroups of $\AGL(1,p)$ on these orbits is of
interest if we intend to apply Lemma \ref{conjugate}, and is the subject of the
following result.
\begin{lemma}\label{action}
Let $H$ be a nontrivial subgroup of $\AGL(1,p)$, acting by conjugation on
itself. Then
\begin{enumerate}
\item if $T\subseteq H$, then for every $i\in\{1,\ldots p-1\}$, $\iota(H)$ fixes
setwise and acts transitively on $m^iT$; and
\item if $T \cap H = 1$ then $H= H_{(\ell, j)}$ for some $\ell, j$ with $\ell$ a
divisor of $p-1$, $\ell > 1$, $0\leq j\leq p-1$, and for each
$i\in\{1,\ldots,p-1\}$, $\iota(H)$ fixes a unique element of the coset $m^iT$,
namely $m^i t^k$ where $k\equiv j(m^i-1)(m^{(p-1)/\ell}-1)^{-1} \pmod{p}$.
\end{enumerate}
\end{lemma}
\begin{proof}
For (i), let $m^i t^j \in m^i  T$. By (\ref{formula}) it follows that $(m^i
t^j)^h\in m^i T$ for all $h\in H$. Also $m^{-i}-1 \not\equiv 0 \pmod{p}$ as
$1\leq i \leq q-1$. Setting $k\equiv -j(1-m^i)^{-1} \pmod{p}$ we have by
(\ref{formula}), $(m^i t^j)^{t^k}=m^i t^{j+k(1-m^i)}=m^i$, and so $ T$ is
transitive on the coset. Thus $H$ fixes setwise and is transitive on $m^i T$.\\
For (ii), if $ T \cap H = 1$ then $H$ is cyclic and equal to $H_{(\ell, j)}=
\<{m^{(p-1)/\ell} t^j}$ for some $\ell, j$ with $\ell > 1$ since $H\neq 1$. An
element $m^i t^k$ of $m^i T$ is fixed under conjugation by $m^{(p-1)/\ell} t^j$
if and only if $(m^i t^k)^{m^{(p-1)/\ell} t^j}=m^i t^k$, which, applying
(\ref{formula}), is equivalent to $k\equiv j(m^i-1)(m^{(p-1)/\ell}-1)^{-1}
\pmod{p}$. Thus $m^i T$ contains a unique element fixed by $H$.
\qed\end{proof}
Recall that $z=m^{(p-1)/q}$. It follows from Lemma \ref{action}(ii) that
$\iota(H_{(\ell,j)})$ fixes a unique element of each orbit $z^iT$ for $0\leq
i\leq q-1$, and these elements form a cyclic subgroup of $G_{pq}$ of order $q$.
\begin{notation}\label{x}
Let $G=G_{pq}$ and $H=H_{(\ell, j)}$ for some divisor $\ell$ of $p-1$, with
$\ell\neq 1$ and $0\leq j\leq p-1$. Let $X$ denote the set of elements of $G$
fixed under conjugation by $H$, so that (by the remarks above) $X=\<{x}$ is a
cyclic subgroup of order $q$, where $x=zt^{j(z-1)(m^{(p-1)/\ell}-1)^{-1}}$. The
cosets of $X$ form a $\rho(G)\iota(H)$-invariant partition of $G$ (recall that
$\rho, \iota$ denote the actions of $G$ on itself by right multiplication and
conjugation respectively), and $G=T\rtimes X$.
\end{notation}
\begin{construction}\label{construction1}
Let $p$ and $q$ be primes with $p\equiv 1\pmod{q}$. Then define the graph
\[
\Gamma(pq):= \Gamma(G_{pq}, T, z),
\]
recalling that $T$ is the translation subgroup of $G_{pq}\leqslant \AGL(1,p)$,
and $\Gamma(G,H,g)=\Cay(G, g^H \cup (g^{-1})^H)$. The graph $\Gamma(pq)$ is
isomorphic to the lexicographic product $C_q[\overline{K_p}]$ if $q$ is odd and
$K_2 [\overline{K_p}] = K_{p,p}$ if $q=2$.
\end{construction}
\begin{construction}\label{construction2}
Let $p$ and $q$ be primes with $q\equiv 1\pmod{p}$, let $\ell$ be a divisor of
$p-1$ such that $\ell > 1$, and let $i$ be an integer with $1\leq i \leq q-1$.
Then define the graph
\[
\Gamma(pq,\ell,i):= \Gamma(G_{pq}, H_{(\ell,1)}, z^i),
\]
recalling that $H_{(\ell,1)}=\<{m^{(p-1)/\ell}t}$, and $\Gamma(G,H,g)=\Cay(G,
g^H \cup g^{-H})$.
\end{construction}
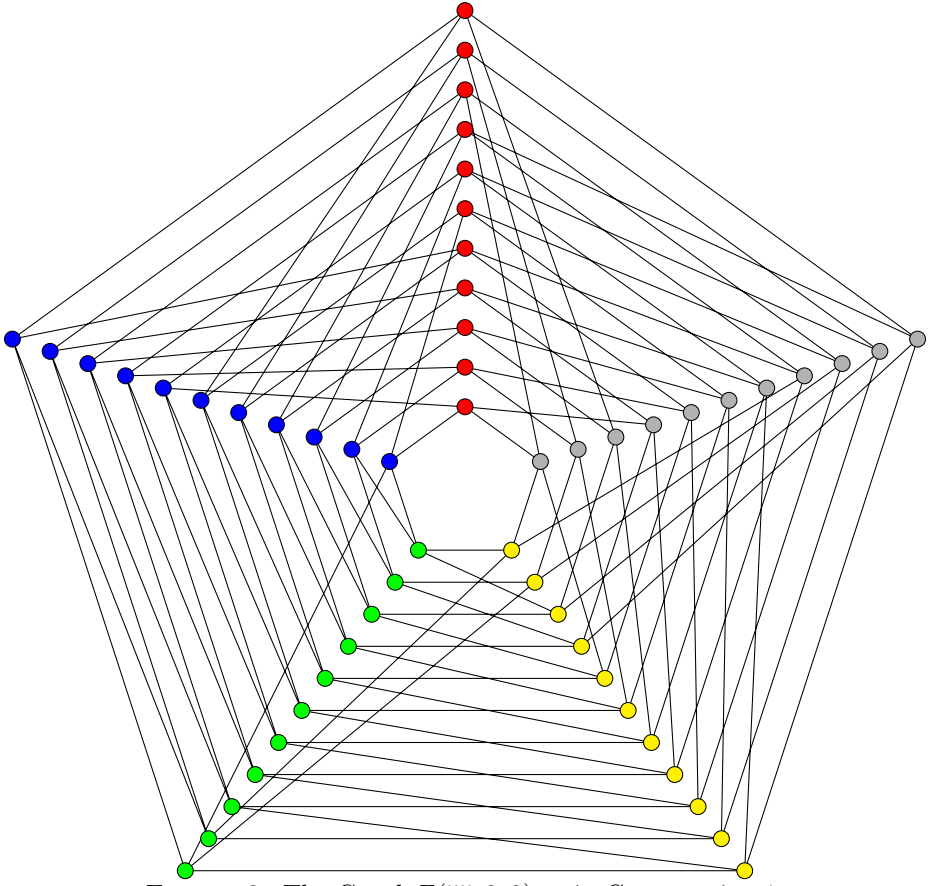
\begin{figure}\label{pic:11522}
\begin{center}
\begin{tikzpicture}[x=15, y=15]
\draw
(18:2) node[normal]{} -- (306:2)
(18:2)                       -- (306:6)

(18:3) node[normal]{} -- (306:3)
(18:3)                       -- (306:7)

(18:4) node[normal]{} -- (306:4)
(18:4)                       -- (306:8)

(18:5) node[normal]{} -- (306:5)
(18:5)                       -- (306:9)

(18:6) node[normal]{} -- (306:6)
(18:6)                       -- (306:10)

(18:7) node[normal]{} -- (306:7)
(18:7)                       -- (306:11)

(18:8) node[normal]{} -- (306:8)
(18:8)                       -- (306:12)

(18:9) node[normal]{} -- (306:9)
(18:9)                       -- (306:2)

(18:10) node[normal]{} -- (306:10)
(18:10)                       -- (306:3)

(18:11) node[normal]{} -- (306:11)
(18:11)                       -- (306:4)

(18:12) node[normal]{} -- (306:12)
(18:12)                       -- (306:5)

(306:2) node[yellow]{} -- (234:2)
(306:2)                       -- (234:11)

(306:3) node[yellow]{} -- (234:3)
(306:3)                       -- (234:12)

(306:4) node[yellow]{} -- (234:4)
(306:4)                       -- (234:2)

(306:5) node[yellow]{} -- (234:5)
(306:5)                       -- (234:3)

(306:6) node[yellow]{} -- (234:6)
(306:6)                       -- (234:4)

(306:7) node[yellow]{} -- (234:7)
(306:7)                       -- (234:5)

(306:8) node[yellow]{} -- (234:8)
(306:8)                       -- (234:6)

(306:9) node[yellow]{} -- (234:9)
(306:9)                       -- (234:7)

(306:10) node[yellow]{} -- (234:10)
(306:10)                       -- (234:8)

(306:11) node[yellow]{} -- (234:11)
(306:11)                       -- (234:9)

(306:12) node[yellow]{} -- (234:12)
(306:12)                       -- (234:10)

(234:2) node[green]{} -- (162:2)
(234:2)                       -- (162:3)

(234:3) node[green]{} -- (162:3)
(234:3)                       -- (162:4)

(234:4) node[green]{} -- (162:4)
(234:4)                       -- (162:5)

(234:5) node[green]{} -- (162:5)
(234:5)                       -- (162:6)

(234:6) node[green]{} -- (162:6)
(234:6)                       -- (162:7)

(234:7) node[green]{} -- (162:7)
(234:7)                       -- (162:8)

(234:8) node[green]{} -- (162:8)
(234:8)                       -- (162:9)

(234:9) node[green]{} -- (162:9)
(234:9)                       -- (162:10)

(234:10) node[green]{} -- (162:10)
(234:10)                       -- (162:11)

(234:11) node[green]{} -- (162:11)
(234:11)                       -- (162:12)

(234:12) node[green]{} -- (162:12)
(234:12)                       -- (162:2)

(162:2) node[blue]{} -- (90:2)
(162:2)                       -- (90:7)

(162:3) node[blue]{} -- (90:3)
(162:3)                       -- (90:8)

(162:4) node[blue]{} -- (90:4)
(162:4)                       -- (90:9)

(162:5) node[blue]{} -- (90:5)
(162:5)                       -- (90:10)

(162:6) node[blue]{} -- (90:6)
(162:6)                       -- (90:11)

(162:7) node[blue]{} -- (90:7)
(162:7)                       -- (90:12)

(162:8) node[blue]{} -- (90:8)
(162:8)                       -- (90:2)

(162:9) node[blue]{} -- (90:9)
(162:9)                       -- (90:3)

(162:10) node[blue]{} -- (90:10)
(162:10)                       -- (90:4)

(162:11) node[blue]{} -- (90:11)
(162:11)                       -- (90:5)

(162:12) node[blue]{} -- (90:12)
(162:12)                       -- (90:6)

(90:2) node[red]{} -- (18:2)
(90:2)                       -- (18:5)

(90:3) node[red]{} -- (18:3)
(90:3)                       -- (18:6)

(90:4) node[red]{} -- (18:4)
(90:4)                       -- (18:7)

(90:5) node[red]{} -- (18:5)
(90:5)                       -- (18:8)

(90:6) node[red]{} -- (18:6)
(90:6)                       -- (18:9)

(90:7) node[red]{} -- (18:7)
(90:7)                       -- (18:10)

(90:8) node[red]{} -- (18:8)
(90:8)                       -- (18:11)

(90:9) node[red]{} -- (18:9)
(90:9)                       -- (18:12)

(90:10) node[red]{} -- (18:10)
(90:10)                       -- (18:2)

(90:11) node[red]{} -- (18:11)
(90:11)                       -- (18:3)

(90:12) node[red]{} -- (18:12)
(90:12)                       -- (18:4)
;
\end{tikzpicture}
\end{center}
\vspace{-6mm}
\caption{The Graph $\Gamma(55,2,2)$ as in Construction \ref{construction1}.}
\end{figure}
\begin{remark}\label{arctrans}
When $q=2$, we have $i=1$ and $z=z^{-1}$. So $H=H_{(\ell,1)}$ acts transitively
on $S = z^H$, and $\Gamma(2p,\ell,1)$ is $\rho(G)\iota(H)$-arc-transitive, by
Lemma \ref{orbits}.
\end{remark}
\begin{lemma}\label{p-1}
Let $\ell=p-1$, and let $\Gamma = \Gamma(pq,\ell,i)$ as defined in Construction
\ref{construction2}. If $q$ is odd, then $\Gamma \cong K_p\times C_q$, and
$\Aut\Gamma = S_p \times D_{2q}$. If $q=2$ then $\Gamma \cong K_p \times K_2$
with $\Aut\Gamma = S_p\times \Z{2}$. In particular, $\Aut\Gamma$ has a system of
imprimitivity consisting of $p$ blocks of size $q$.
\end{lemma}
\begin{proof}
Set $H = H_{(p-1,1)} = \<{mt}$, so $\Gamma = \Gamma(G,H,z^i)$. A vertex in
$\Gamma=\Gamma(pq,p-1,i)$ is joined to the identity if and only if it is
contained in the set $S=z^{iH}\cup z^{-iH}= (Tz^i \cup Tz^{-i})\setminus X$.
Similarly $g\in G$ is joined to precisely $(Tz^ig \cup Tz^{-i}g)\setminus
Xg$.\\\\
Consider the two partitions $\mathscr{P}_T = \{Tg\mid g\in G\}$ and
$\mathscr{P}_X = \{Xg\mid g\in G\}$. The quotient graphs
$\Gamma_{\mathscr{P}_T}$ and $\Gamma_{\mathscr{P}_X}$ are isomorphic to $C_q$
and $K_p$ respectively, and two vertices in $\Gamma$ are joined precisely when
the corresponding vertices in the quotient graphs are joined. This is the
definition of $K_p\times C_q$ (see Definition \ref{directproduct}). Thus
$\Aut\Gamma \geqslant \Aut K_p \times \Aut C_q = S_p \times D_{2q}$. It is not
difficult to prove that equality holds, and so $\mathscr{P}_X$ is
$\Aut\Gamma$-invariant with blocks of size $q$.
\qed\end{proof}
\subsection{Normal edge-transitive Cayley graphs for $G_{pq}$}\label{nonabelian}
\subsubsection{Proof of Proposition \ref{classification}}
We divide the connected, normal edge-transitive Cayley graphs for $G_{pq}$ into
two distinct classes. From now on we assume that $G=G_{pq}= \<{t,z}$ as in
Section \ref{affine}, that $H\leqslant \Aut G=\iota(\AGL(1,p))$, and that
$N=\rho(G)\iota(H)$ acts edge-transitively on a connected Cayley graph
$\Gamma=\Gamma(G,H,g)=\Cay(G,S)$ where $S=g^H \cup g^{-H}$ for some $g\in
G\setminus\{1\}$. Recall that $T$ is the unique $\iota(H)$-invariant normal
subgroup of $G$ and that, by Example \ref{prime}, $\Gamma_T=\Gamma(q,a)$ for
some $a\mid (q-1)$ with $a$ even if $q>2$.\\\\
If $ T\subseteq H \leqslant G$ then, by Lemma \ref{action}(i), for some $i\in
\{1,\ldots,p-2\}$, the Cayley graph $\Gamma(G,H,g)=\Cay(G,S)$ with $S= m^i  T
\cup m^{-i} T$, where $g\in m^i T$ and $g^H=m^i T$. Hence, by Lemma
\ref{cosets}, $\Gamma\cong \Gamma_{ T}[\overline{K_p}]$. The quotient graph
$\Gamma_{T}$ is $K_2 =\Gamma(2,1)$ if $q=2$, or $C_q = \Gamma(q,2)$ if $q$ is
odd (since it has valency two). Moreover as $g^H=m^i T$ it follows from
Proposition \ref{orbits} that $\Gamma$ is normal edge-transitive relative to
$N$. Thus we have proved the following.
\begin{lemma}\label{lexcase}
If $p$ divides $|H|$ then $\Gamma\cong\Gamma_T[\overline{K_p}]$, where
$\Gamma_T=\Gamma(q,2)\cong C_q$ if $q$ is odd, or $\Gamma(2,1)= K_2$ if $q=2$,
and $\Gamma$ is normal edge-transitive relative to $\rho(G)\iota(H)$.
\end{lemma}
Note that this Lemma shows that all assertions of Proposition
\ref{classification}
hold if $p$ divides $|H|$, and in this case $\Gamma$ is as in Theorem
\ref{intro}(i). Also this Lemma implies $\Gamma(G,H,z)=\Gamma(G,T,z)$,
and so if $H$ contains $T$ we assume without loss of generality that $H=T$. We
now consider the second case where $H\cap T=1$, and hence $|H| \mid (p-1)$.
\begin{lemma}\label{conj}
Let $H=H_{(\ell, j)}\subseteq \AGL(1,p)$ (with $\ell > 1, \ell \mid p-1$), and
let $g \in G:=G_{pq}$, and suppose that $\Gamma= \Gamma(G,H,g)$ is connected.
If $\ell = |H|$ divides $p-1$ then $\ell > 1$ and $\Gamma \cong
\Gamma(pq,\ell,i)$ as in Construction \ref{construction2} for some $i$.
\end{lemma}
\begin{proof}
By Lemma \ref{action}, $H=H_{(\ell,j)}$ for some divisor $\ell$ of $p-1$ and
some $j$ where $0\leq j\leq p-1$, and we have $g = z^i t^k$ for some $i,k$.
Using Equation \eqref{formula} it is straightforward to check that $y_1 :=
t^{k(m^{i(p-1)/q}-1)^{-1}}$ conjugates $g$ to $z^i$ and conjugates $H_{(\ell,
j)}$ to $H_{(\ell, j')}$ for some $j'$. Since $\Gamma$ is connected $\ell = |H|
> 1$. If $j'\neq 0$ there exists $r$ such that
$j'm^r\equiv 1 \pmod{p}$ (interpreting $m$ here as an element of $\Z{p}$) and
hence such that $(m^{(p-1)/\ell}t^{j'})^{z^r}=m^{(p-1)/\ell}t$. Thus
$H_{(\ell,j)}^{\iota(y_1 m^r)}=H_{(\ell,j')}^{\iota(m^r)}=H_{(\ell,1)}$, and
$g^{\iota(y_1 m^r)} = (z^i)^{m^r} = z^i$ (since $z^i \in \<{m}$)). So $\Gamma
\cong \Gamma(G, H_{(\ell,1)},z^i) = \Gamma(pq,\ell,i)$.\\\\
If $j'=0$ then $H_{(\ell,0)}$ centralises $z^i$ and hence $\Gamma \cong
\Gamma(G,H_{(\ell,0)}, z^i)=\Cay(G,S)$ where
$S=\{z^i, z^{-i}\}$ and so by the remarks in Section \ref{connectivitysection},
$\Gamma$ is isomorphic to $p.C_q$, contradicting the connectivity of $\Gamma$.
\qed\end{proof}
We now complete the proof of Proposition \ref{classification}. By the remarks
following Lemma \ref{lexcase} we may assume that $H\cap T = 1$, and by Lemma
\ref{conj}, we may further assume that $\Gamma = \Gamma(G_{pq},
H_{(\ell,1)},z^i) = \Gamma(pq,\ell,i)$ for some divisor $\ell$ of $p-1$
with $\ell \neq 1$ and with $1\leq i\leq q-1$. If $q$ is odd then
$\Gamma(G,H_{(\ell,1)},z^i)=\Gamma(G,H_{(\ell,1)},z^{q-i})$,
since $(z^{q-i})^H\cup (z^{-{q-i}})^H=(z^{-i})^H\cup (z^{i})^H$. So if $q$ is
odd we may assume $1\leq i \leq \frac{q-1}{2}$. If $q=2$, then $i=1$.\\\\
If $\ell = p-1$, then by Lemma \ref{p-1}, $\Gamma$ is as in Theorem
\ref{intro}(ii). In all other cases (that is, $1 < \ell < p-1$), $\Gamma$ is as
in Theorem \ref{intro}(iii): the vertices
of $\Gamma_T$ are the cosets of $T$, and in all cases $S = z^iH \cup z^{-i}H$,
and so $ST = z^iT \cup z^{-i}T$. Thus the connection set of $\Gamma_T = ST/T =
\{z^iT, z^{-i}T\}$, which has size $1$ if $q=2$ and $2$ if $q$ is odd. Thus
$\Gamma_T = K_2$ if $q=2$ and $C_q$ when $q$ is odd. Proposition
\ref{classification} is now proved.
\subsubsection{Cayley graphs for $\rho(T)\times\lambda(X)$}
The graphs of case (iii) all seem essentially `the same' at first glance.
However, the structure of the graph differs fundamentally depending on the
parameter $\ell$. This is because we sometimes, but not always, have a regular
abelian subgroup of $\Aut\Gamma$ (see Lemma
\ref{L} below). In this case $\Gamma$ may be reinterpreted as a Cayley graph for
an
abelian group. Recall from Sections \ref{background} and \ref{permutations} that
$N=\rho(G)\Aut(G)_S$ is the normaliser of $\rho(G)$
in $\Aut\Gamma$, and $\lambda(G)$ denotes the left regular action $\lambda_g : x
\mapsto g^{-1} x$ of $G$.
\begin{lemma}\label{L}
Let $\Gamma=\Gamma(p,q,\ell,i)$, and suppose $q$ divides $\ell$. Then the
following hold, for $X$ as in Notation \ref{x}:
\begin{enumerate}
\item $X\leqslant H$;
\item $\lambda(X) \leqslant N$;
\item $\rho(T)\times \lambda(X)$ is a regular subgroup of $\Aut\Gamma$, and so
$\Gamma$ is a Cayley graph for $\rho(T)\times\lambda(X)=\Z{p}\times\Z{q}$;
\item As a Cayley graph for $\rho(T)\times\lambda(X)$, $\Gamma$ is normal
edge-transitive.
\end{enumerate}
\end{lemma}
\begin{proof}
Part (i) follows from the definition of $X$. For part (ii), observe that
$\lambda(X) \leqslant \rho(X)\iota(X) \leqslant \rho(G)\iota(H)=N$. Part (iii)
then follows easily.\\
For (iv), note that since $\rho(T)\chr\rho(G)\triangleleft N$, we have
$\rho(T)\triangleleft N$, and since $X$ is by definition fixed under $\iota(H)$,
so is $\lambda(X)$. So $\lambda(X)$ is centralised by both $\iota(H)$ and
$\rho(G)$, it is normal in the product, and so $\rho(T)\lambda(X) \triangleleft
N$. Thus $N\subseteq N_{\Aut\Gamma}(\rho(T)\lambda(X))$, and so the normaliser
is transitive on $\EGamma$.
\qed\end{proof}
When $q|\ell$ there is only one graph up to isomorphism: different choices of
$i$ give isomorphic graphs. 
%
\begin{proposition}\label{qdividesl}
If $q$ divides $\ell$, then $\Gamma(pq,\ell,i) \cong \Gamma(\Z{p}\times \Z{q},
\hat{H}, (1,1))$ where $\hat{H} = H(\frac{q-1}{2}, \frac{p-1}{\ell}, d)$ as in
Definition \ref{AbelianCaseConditionsDef}, and 
\[
d = 
\begin{cases} 
0 					& \text{ if $\ell$ is even; and}\\
\frac{p-1}{2\ell}   & \text{ if $\ell$ is odd.}
\end{cases}
\]
In particular $\Gamma(pq,\ell,i)$ is independent of $i$ up to isomorphism.
\end{proposition}
\begin{proof}
By Lemma \ref{L}, $L=\rho(T)\times\lambda(X)\leqslant \Aut\Gamma$. Recall that
$N = \rho(G)\iota(H) = N_{\Aut \Gamma}(\rho(G))$. Since $\rho(T)$ is a
characteristic subgroup of $\rho(G)$, we have that $\rho(T)$ is normalised by
$N$. Since $\lambda(X)$ is centralised by $\rho(G)$ (the left and right regular
actions centralise one another) and is centralised by $\iota(H)$ (since
$X\leqslant H$ and $H$ is cyclic), we have that $\lambda(X)$ is centralised by
$N$. It follows that $N \leqslant N_{\Aut \Gamma}(L)$. \\\\
Now since (by Lemma \ref{L}) $\Gamma$ is a Cayley graph for $L$, we may identify
the vertices of $\Gamma$ with the elements of $L$, and $\Gamma \cong \Cay(L,S)$ for some $S \subseteq L$ with $S = S^{-1}$. Then the automorphism
$\varphi: x \to x^{-1}$ is an automorphism of $L$ (since $L$ is abelian) and it
fixes the connection set $S$, and so $\varphi \in N_{\Aut \Gamma } (L)$. Set $\hat{N} = \<{N,
\varphi}$, and set $\hat{H} = \hat{N}_{1_G} = \<{\iota(H), \varphi}$. \\\\
Now $\Gamma$ is normal edge-transitive as a Cayley graph for $L$ (since $N\leqslant N_{\Aut\Gamma}(L)$ and $N$ is edge-transitive). So $\hat{H}$ is transitive on the connection set $S$ (since $\varphi$ switches an
element with its inverse). Moreover, we have $|\hat{H}| = |S| = 2\ell$, and $\hat{N}$
is transitive on the arcs of $\Gamma$.\\\\
We seek to determine the parameters $d_2,d_1,d$ as in Definition
\ref{AbelianCaseConditionsDef}, such that $\hat{H} \cong \<{(x^d, y^{d_2}),
(x^{d_1},1)} \leqslant \Z{p}^* \times \Z{q}^* \cong \Aut(\rho(T)) \times \Aut
(\lambda(X))$ (as in Theorem \ref{zpzqthm}). Since $\iota(H)$ centralises $\lambda(X)$, the subgroup of
$\Aut(\lambda(X))$ induced by the action of $\hat{H}$ is isomorphic to $\Z{2}$,
and so $d_2 = \frac{q-1}{2}$. Since $\iota(H)$ acts faithfully on $\rho(T)$, we
have $\hat{H} \cap (\Aut(\rho(T)) \times 1) = \iota(H)$, and so $d_1 =
\frac{p-1}{\ell}$.\\\\
Suppose that $d > 0$. Then the conditions \ref{AbelianCaseConditionsDef} give $0 < d \leq d_1$ and $d_1 \mid 2d$ which implies that $d_1 = 2d$ is even. Thus $d = \frac{p-1}{2\ell}$. In this case the first generator of $\hat{H} = H(d_2,d_1,d)$ is $(x^{(p-1)/2\ell}, -1)$, which squares to the second generator $(x^{d_1},1)$. Thus $\hat{H}$ is cyclic and so $\varphi = (-1,-1)$ is the unique involution in $\hat{H}$. The only element in $\hat{H}$ with first entry $-1$ is $(x^{(p-1)/2\ell},-1)^{\ell} = (-1,(-1)^{\ell})$, and it follows that $\ell$ is odd.\\\\
Suppose now that $d=0$. Then $\hat{H}$ contains $(1,-1)$ and since $\hat{H}$ also contains $\varphi = (-1,-1)$, we have $(-1,-1)\in \hat{H} \cap (\Aut(\rho(T))\times 1) = \<{(x^{d_1},1)}$. This implies that $\ell = |x^{d_1}|$ is even.
\qed\end{proof}
In the case $q\nmid\ell$, however, the graphs in Proposition
\ref{classification}
are pairwise nonisomorphic (Corollary \ref{nonisomorphic} below).
\section{Redundancy and Automorphisms}\label{automorphism}
In this section we discuss the possibility of redundancy in our classification,
that is, isomorphisms between the graphs $\Gamma(pq,\ell,i)$ for different
choices of parameters. In doing so we determine the automorphism groups of our
graphs.
Recall the following: $p$ and $q$ are primes with $q$ dividing $p-1$, $\ell$ is
a proper divisor of $p-1$ with $\ell > 1$, $i$ is an integer with $1\leq i \leq
(q-1)/2$, $G=G_{pq}$, $H=H_{(\ell, 1)} = \<{m^{(p-1)/\ell}t}$, and
$N=\rho(G)\iota(H)$. Define $\Gamma(pq,\ell,i)=\Gamma(G,H,z^i)$, and
$Y=\Aut\Gamma$. In this section we determine $Y$ for most values of
$(p,q,\ell,i)$ (see Theorem \ref{autgamma}), and decide when different sets of
parameters yield isomorphic
graphs (Corollary \ref{nonisomorphic}).\\\\
It is obvious that different primes $p$ and $q$ generate nonisomorphic graphs,
as $\Gamma(pq,\ell,i)$ has $pq$ vertices. Each graph $\Gamma(pq,\ell,i)$ has
valency $\ell$ or $2\ell$, according as $q$ is odd or even. Thus different
choices for $\ell$ also yield nonisomorphic graphs. We therefore need only
decide
whether $\Gamma(pq, \ell, i)\cong \Gamma(pq, \ell, i')$ implies $i=i'$. 
\begin{theorem}\label{autgamma}
Let $\Gamma=\Gamma(pq,\ell,i)$ as defined in Construction \ref{construction2},
and let $Y=\Aut\Gamma$. Then
 \[
Y=\begin{cases} 
\rho(G).\iota(H) & \text{when $q=2$ or $q \nmid \ell$ and $\ell < p-1$;} \\
\rho(G).\iota(H).\Z{2} & \text{when $q\geq 3$, $q \mid \ell$ and $\ell <
p-1$;}\\
S_p \times \Z{2} & \text{when $\ell=p-1$ and $q=2$; and}\\
S_p \times D_{2q} & \text{when $\ell=p-1$ and $q=3$;}
\end{cases}
\]
except in the cases $(p,q,\ell,i)=(7,3,2,1), (7,2,3,1), (11,2,5,1)$ and
$(73,2,9,1)$.
\end{theorem}
We prove Theorem \ref{autgamma} over the course of this section. First we give a proof of Theorem \ref{intro}.
\begin{proof}[Proof of Theorem \ref{intro}]
That $\Gamma$ satisfies one of Theorem \ref{intro}(i)-(iii) follows from Proposition \ref{classification}, and the structure of $Y=\Aut\Gamma$ follows from Theorem \ref{autgamma} in all cases except the four exceptional parameter sets of Theorem \ref{autgamma}. The other automorphism groups can be calculated manually (using, for example, GAP).
\qed\end{proof}
\afterpage{
\begin{figure}[H]
\begin{center}
\begin{tikzpicture}[scale=0.27]
\draw \foreach \x in {0}{
\foreach \y in {0,2,6,14,30,35,53,62,72}{
(0, 72-\x) node[biggraph]{} -- (40,72-\x-\y) node[biggraph]{}
}}
;
\draw \foreach \x in {1,2,...,10}{
\foreach \y in {0,2,6,14,30,35,53,62,-1}{
(0, 72-\x) node[biggraph]{} -- (40,72-\x-\y) node[biggraph]{}
}}
;
\draw \foreach \x in {11,12,...,19}{
\foreach \y in {0,2,6,14,30,35,53,-11,-1}{
(0, 72-\x) node[biggraph]{} -- (40,72-\x-\y) node[biggraph]{}
}}
;
\draw \foreach \x in {20,21,...,37}{
\foreach \y in {0,2,6,14,30,35,-20,-11,-1}{
(0, 72-\x) node[biggraph]{} -- (40,72-\x-\y) node[biggraph]{}
}}
;
\draw \foreach \x in {38,39,..., 42}{
\foreach \y in {0,2,6,14,30,-38,-20,-11,-1}{
(0, 72-\x) node[biggraph]{} -- (40,72-\x-\y) node[biggraph]{}
}}
;
\draw \foreach \x in {43,44,...,58}{
\foreach \y in {0,2,6,14,-43,-38,-20,-11,-1}{
(0, 72-\x) node[biggraph]{} -- (40,72-\x-\y) node[biggraph]{}
}}
;
\draw \foreach \x in {59,60,...,66}{
\foreach \y in {0,2,6,-59,-43,-38,-20,-11,-1}{
(0, 72-\x) node[biggraph]{} -- (40,72-\x-\y) node[biggraph]{}
}}
;
\draw \foreach \x in {67,68,...,70}{
\foreach \y in {0,2,-67,-59,-43,-38,-20,-11,-1}{
(0, 72-\x) node[biggraph]{} -- (40,72-\x-\y) node[biggraph]{}
}}
;
\draw \foreach \x in {71,72}{
\foreach \y in {0,-71,-67,-59,-43,-38,-20,-11,-1}{
(0, 72-\x) node[biggraph]{} -- (40,72-\x-\y) node[biggraph]{}
}}
;
\end{tikzpicture}

\end{center}
\caption{The graph $\Gamma(73.2,9,1)$, an incidence graph of a projective plane and an exception in Theorem \ref{autgamma}.}\label{big projective graph}
\end{figure}
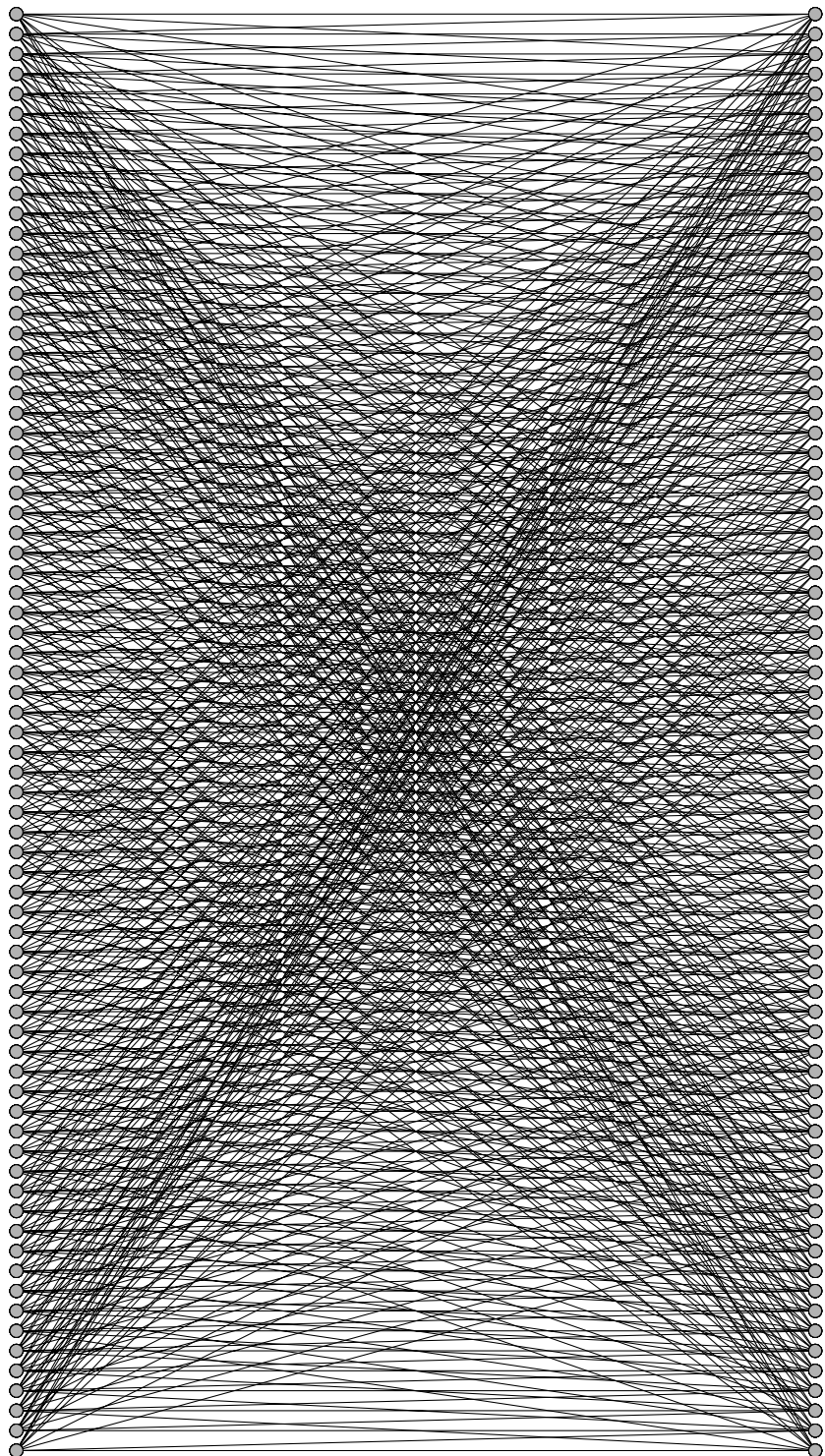
}
Next we deduce from Theorem \ref{autgamma} our claim about graph isomorphisms.
\begin{corollary}\label{nonisomorphic}
Let $\Gamma(pq, \ell, i),\Gamma(pq,\ell,i')$ be defined as in Construction
\ref{construction2}, and suppose that $q\nmid \ell$. Then
$\Gamma(pq,\ell,i)\cong\Gamma(pq,\ell,i')$ if and only if $i=i'$.
\end{corollary}
\begin{proof}
Let $\Gamma=\Gamma(pq,\ell,i),\Gamma'=\Gamma(pq,\ell,i')$. If
$q\leq 3$ then $i=i'=1$, and so we assume without loss of generality that $q\geq
5$. An
isomorphism $\varphi:\Gamma\to\Gamma'$ is a permutation of $G$ such that
$\EGamma^{\varphi}=\EGamma'$. We may assume without loss of generality that
$\varphi$ fixes the identity of $G$, since both graphs are
vertex-transitive.\\\\
Now $\varphi^{-1}\Aut\Gamma\varphi=\Aut\Gamma'$, but by Theorem \ref{autgamma},
$\Aut\Gamma=\Aut\Gamma'=\rho(G)\iota(H)$, and so
$\varphi^{-1}(\Aut\Gamma)\varphi=\Aut\Gamma$. Hence $\varphi\in N_{1_G}=(N_{\Sym
G}(\rho(G)\iota(H)))_{1_G}$.\\\\
Since $N$ normalises $\rho(G)$, it is contained in the holomorph of $G$, namely
$\rho(G).\Aut(G)$, and hence $N_{1_G} \subseteq \Aut(G)$. But by Formula
(\ref{formula}), the action of $\Aut(G)$ fixes the cosets of $T$ setwise. Now an
isomorphism fixing $1_G$ must map $(z^{i})^{H}\cup (z^{-i})^{H}$ to
$(z^{i'})^{H}\cup
(z^{-i'})^{H}$, and if $q \geq 5,1\leq i\leq (q-1)/2$ then this is possible only
if
$i=i'$, as $(z^{i})^{H}\subseteq z^iT$.
\qed\end{proof}
Our first step in the proof of Theorem \ref{autgamma} is to identify one of the
exceptional cases.
\begin{proposition}\label{primitive}
Suppose $\Gamma=\Gamma(pq,\ell,i)$ is vertex-primitive. Then $(p,q,\ell,i)=(7,3,
2,1)$ and $\Gamma$ is the flag graph $\Gamma_F$ of the Fano Plane (see Section
\ref{fano}, Figure \ref{primitive picture}), with automorphism group $\PGL(3,2).\Z{2}$.
\end{proposition}
\begin{proof}If $q=2$ then $\Gamma$ is bipartite, and the two parts form a
system of imprimitivity. Also if $\ell = p-1$ then $\Aut\Gamma$ is imprimitive
by Lemma \ref{p-1}. Thus $q\geq 3$, $p \geq 7$ since $q\mid p-1$, and $\ell$ is
a proper divisor of $p-1$. The edge-transitive, vertex-primitive graphs of order
a product of 2 primes are classified in \cite[Table I, Table
III]{praeger1993vertex},
along with their valency and whether or not they are Cayley graphs. Requiring
that $\Gamma$ be a Cayley graph, and that $q$ and $\ell=\frac{\val\Gamma}{2}$
are divisors of $p-1$ with $1 < \ell < p-1$, we are left with the possibilities
in Table \ref{Table:possibilities}:
\begin{table}
\begin{tabular}[h]{|l|c|c|c|}
\hline
$\Soc(\Aut\Gamma)$ & p  & q               & $\ell=\frac{\val\Gamma}{2}$ \\ 
\hline
$A_p$              & p  & $\frac{p-1}{2}$ & $\frac{(p-2)(p-3)}{4}$ \\
$\PSL(2,11)$       & 11 & 5               & 2 \\
$\PSL(2,23)$       & 23 & 11              & 2 \\
$\PSL(2,p)$        & p  & $\frac{p-1}{2}$ & $\frac{p+1}{8},\frac{p+1}{4},
\frac{p+1}{2}$ \\
\hline
\end{tabular}
\caption{Possibilities in the proof of Proposition
\ref{primitive}}\label{Table:possibilities}
\end{table}
The fact that $\ell > 1$ and $\ell$ divides $p-1$ rules out line 1 and in line 4
implies that $p=7$, $\ell=2$ and $q=3$. Thus we have exactly three
$(p,q,\ell)$ to check further. Using Nauty \cite{nauty} and the package GRAPE
\cite{grape} for GAP \cite{GAP4}, we constructed the graphs $\Gamma(pq, \ell,i)$
for the three remaining possible $p,q,\ell$ as in the table (and for every
$i\leq (q-1)/2$) and computed their automorphism groups, finding that the
automorphism group acts imprimitively for the graphs in lines 2 and 3 and that
the graph $\Gamma(21,2,1)$ is vertex-primitive and is the flag graph of the Fano
plane as asserted.
\qed\end{proof}
\subsection{Main Case: $\Aut\Gamma$ is imprimitive}
By Proposition \ref{primitive}, if $(p,q,\ell)\neq (7,3,2)$ then $\Aut\Gamma$ is
imprimitive. Throughout this Section suppose $\Gamma=\Gamma(pq,\ell,i)$ as
defined in Construction \ref{construction2}, with either $(q,i)=(2,1)$ or $q$
odd, $1\leq i\leq (q-1)/2$ and $(p,q,\ell)\neq (7,3,2)$, and let $Y=\Aut\Gamma$.
The case $\ell=p-1$ has been dealt with in Lemma \ref{p-1}, so we assume $1<\ell
< p-1$. By our construction we know that $N:=\rho(G)\iota(H) \leqslant Y$; thus
any $Y$-invariant partition of $\VGamma$ is also $N$-invariant. The following
lemma describes the only $N$-invariant partitions, and so the only feasible
$Y$-invariant partitions.
\begin{lemma}\label{partition}
Let $\Gamma=\Gamma(pq,\ell,i)$ with $\ell < p-1$, and let $N=\rho(G).\iota(H)$.
Then the following are the only nontrivial $N$-invariant partitions of
$\VGamma$:
\begin{enumerate}
\item The partition of $G$ into the right cosets of $T= \<{t}$, consisting of
$q$ blocks of size $p$; and
\item The partition into the right cosets of the subgroup $X$ (see Notation
\ref{x}). consisting of $p$ blocks of size $q$.
\end{enumerate}
\end{lemma}
\begin{proof}
Let $B$ be a block of imprimitivity for $N$ containing $1_G$. By \cite[Theorem
3(a)]{praeger1999finite}, $B$ is a subgroup of $G$. The setwise stabiliser of
$B$ in
$\rho(G)$ is $\rho(B)$, and is a normal subgroup of $N_B$. Since $\iota(H) =
N_{1_G}$ leaves $B$ invariant (since $1_G \in B$), it follows that $B$ is
$H$-invariant. Conversely each $H$-invariant subgroup of $G$ is a block for
$N$.\\\\
Since $T$ is normal in $\AGL(1,p)$ and $H \subseteq \iota(\AGL(1,p))$, $T$ is
$H$-invariant and so the cosets of $T$ form an $N$-invariant partition as in
(i). Any $H$-invariant subgroup $L$ of $G$ with
$L\neq T$ has order $q$, and since by Formula \eqref{formula} $H$ fixes each
coset of $T$ setwise, $H$ must centralise $L$. Thus by Lemma \ref{action} there
is only one other $H$-invariant subgroup, namely the subgroup $X= \<{x}$, where
$x= m^i t^{j(m^i-1)(m^{(p-1)/\ell}-1)^{-1}}$ (see Notation \ref{x}), as in (ii).
\qed\end{proof}
This gives us two possibilities for $Y$-invariant partitions of the vertex set:
one into $p$ blocks of size $q$ and the other into $q$ blocks of size $p$. We
prove the following lemma in the course of the section:
\begin{lemma}\label{casetwo}
If $(p,q,\ell,i)\neq (7,3,2,1)$ and $1<\ell < p-1$, then the cosets of $T$ form
a $Y$-invariant partition of $\VGamma$.
\end{lemma}
We begin by noting that in the case $q=2$, the cosets of $T$ form a bipartition
of $\Gamma$, and hence a system of imprimitivity. Now we assume: 
\begin{equation}\label{assumption}
\text{$q\geq 3$ and the cosets of $X$ form a system of imprimitivity $\mathscr{P}$ for $Y$.}
\end{equation}
If \eqref{assumption} does not hold, then by Lemma \ref{partition} the result is
proved.
\begin{lemma}\label{howmany}
Assume \eqref{assumption} holds, and let $s\in S=(z^{i})^{H} \cup (z^{-i})^{H}$.
Then $|S \cap Xs|\in \{1,2\}$ and is independent of the choice of $s$.
\end{lemma}
\begin{proof}
Note first that $s\in S \cap Xs$, and so $|S\cap Xs|\geq 1$. Now suppose
$|S\cap Xs| \geq 3$. Then since $H$ has just two orbits on $S$, there exist
distinct $s_1,s_2\in S\cup Xs$ with $s_1^h=s_2$ for some $h\in H$. Since
$s_1,s_2\in Xs$, we have $s_1 s_2^{-1}\in X$, but on the other hand
$s_1s_2^{-1}=s_1s_1^{-h}$, and since $H$ fixes the cosets of $T$ setwise it
follows that $s_1s_2^{-1}\in T$. But $T\cap X = \{1\}$, and so $s_1=s_2$,
contradiction.\\\\
If $|S \cap Xs| = 1$ for each $s\in S$ there is nothing more to prove, so
suppose that $S\cap Xs = \{s, s'\}$ with $s\neq s'$.
Then by the above argument $s'$ cannot be in $s^H$, and so $s'\in (s^{-1})^H$.
So there exists $h\in H$ with $s'=s^{-h}$. Now choose $s_2\in S$; then $s_2$ is
in the $H$-orbit of either $s$ or $s'$. Suppose $s_2\in s^H$. Then for some
$h'\in H$, $s_2=s^{h'}$. Then $s_2 s_2^{h} = s^{h'}s^{h'h} = (ss^h)^{h'} \in X^H
= X$, so $Xs_2 = X s_2^{-h}$ and so $|S\cap Xs_2| =2$ for every $s_2\in S$. If
$s_2 \in (s')^H$ then the same argument holds with $s'$ in place of $s$.
\qed\end{proof}
We now investigate the structure of the kernel $K=Y_{(\mathscr{P})}$ and its
action on each member of the partition $\mathscr{P}$ of \eqref{assumption}. We assume for the moment that $K$ is
nontrivial. In this case, $K$ is transitive on every block (as $Y$ acts
primitively on each block and $K \triangleleft Y$, Lemma \ref{blocks} shows $K$
is transitive). So if $K\neq 1$, the $K$-orbits are the cosets of $X$. Moreover,
since $K$ acts transitively on each block $Xg$ and each block has prime size
$q$, by Lemma \ref{affineortwotransitive}, $K^{Xg}$ is primitive.
\begin{lemma}\label{actiononotherblocks}
Assume \eqref{assumption} holds. Then the pointwise stabiliser $K_{(X)}$ is
trivial, and so $K \cong K^{X}$.
\end{lemma}
\begin{proof}
If $K=1$ there is nothing to prove so assume $K\neq 1$. Let $s\in S$. Since
$K_{(X)}$ fixes $1_G \in X$ it follows that $K_{(X)}$ fixes $S$ setwise. Also
$K_{(X)} < K = Y_{(\mathscr{P})}$ fixes the block $X$ setwise and hence
$K_{(X)}$ fixes $S \cap Xs$ setwise. By Lemma \ref{howmany}, $|S \cap Xs| \leq 2
< q = |Xs|$, and hence $K_{(X)}$ is not transitive on $Xs$. Since $K$ is
primitive on $Xs$, its normal subgroup $K_{(X)}$ must therefore act trivially on
$Xs$, and since this holds for all $s\in S$, it follows by connectivity that
$K_{(X)} = 1$.
\qed\end{proof}
\begin{lemma}\label{K1}
Assume \eqref{assumption} holds, and that $K\neq 1$. Then either $K=\lambda(X)$
or $K=
\lambda(X)\rtimes\Z{2}\cong D_{2q}$, and in particular, $\lambda(X)\triangleleft
Y$.
\end{lemma}
\begin{proof} 
Let $s\in S$. Then $s\in S \cap Xs$ and by Lemma \ref{howmany},
$|S\cap Xs|\leq 2$. Suppose first that $S\cap Xs=\{s\}$. Then
$K_1$ fixes $S \cap Xs$ and so $K_1 \leqslant K_{s}$. Since all $K$-orbits have
the same length, $K_1=K_s$, and this holds for every $s\in S$. By connectivity,
$K_1=1$, and so $|K|=q$.\\\\
Now suppose $S\cap Xs = \{s,s'\}$. Since $K_1$ fixes $S\cap Xs$
setwise it follows that $|K_1 : K_{1,s} | \leq 2$ and $K_{1,s} \subseteq
K_{s,s'}$. Thus $|K_s : K_{s,s'}| \leq 2$ and in particular if $K^{Xs}$ is
2-transitive then $q=3$ and $K^{Xs} = S_3 = \AGL(1,3)$. Thus by Lemma
\ref{affineortwotransitive}, in all cases $K^{Xs} \leqslant \AGL(1,q)$ and
$K_{s,s'}$ fixes $Xs$ pointwise. We therefore have $|K^{Xs}| = q|K_s : K_{s,s'}|
\leq 2q$.\\\\
Thus $|K|$ is either $q$ or $2q$, and so $K$ has a
characteristic subgroup $K_0 \cong \Z{q}$ and $K_0 \triangleleft Y$. We claim
that $K_0=\lambda(X)$.
Consider the subgroup $Y_0:=\<{K_0, \rho(T)}$ of $\Sym(G)$. Now $\rho(T) \cap
K_0 =1$ and $\rho(T)$ normalises $K_0$ and hence $|Y_0|=pq$. Since $p > q$,
$\rho(T)$ is a normal subgroup of $Y_0$, and so $Y_0 = K_0 \times \rho(T)$ and
$\rho(T)\leqslant C_{\Sym G}(K_0)$.\\\\
Now consider $\<{\rho(X), K_0}$. This group has order $q^2$ and so is abelian.
In particular, $\rho(X)\subseteq C_{\Sym G}(K_0)$, and so $\rho(G) \leqslant
C_{\Sym G}(K_0)$, as $\rho(G)=\<{\rho(T),\rho(X)}$. This implies that
$K_0\leqslant C_{\Sym G}(\rho(G))=\lambda(G)$.
So $K_0=\lambda(X')$ for some subgroup $X'$ of $G$ of order $q$, and since $K_0$
fixes $X$ setwise we must have that $X'=X$. If $K=\lambda(X)\rtimes \Z{2}$ then
$K\cong D_{2q}$ as it cannot possibly be cyclic (all of its orbits have size
$q$).
\qed\end{proof}
This yields three cases, according to $K$: it is either $D_{2q}, \Z{q}$ or $1$.
\begin{lemma}\label{KisDihedral}
Assume \eqref{assumption} holds. If $K\cong D_{2q}$ then the conclusion of Lemma
\ref{casetwo} holds.
\end{lemma}
\begin{proof}
By Lemma \ref{blocks}(ii), $\Fix K_1$ is a block of imprimitivity for $Y$ in
$\VGamma$. If $K\cong D_{2q}$ then by Lemma \ref{actiononotherblocks}, $K$ acts
faithfully as $D_{2q}$ on every block in $\mathscr{P}$, and so $K_1\cong\Z{2}$
fixes a unique point in each of the $p$ blocks. By Lemma \ref{partition}, $\Fix
K_1$ must be a coset of $T$ and Lemma \ref{casetwo} is proved in this case.
\qed\end{proof}
Thus we may assume that $K=1$ or $K=\lambda(X)$. We consider these cases
separately, investigating the quotient graph $\Gamma_{\mathscr{P}}$ and the
group $Y^{\mathscr{P}}\cong Y/K$.
\begin{lemma}\label{kis1}
Assume \eqref{assumption} holds. If $K=1$ then the conclusion of Lemma
\ref{casetwo} holds.
\end{lemma}
\begin{proof}
Suppose that $K=1$. Then $Y\cong Y^{\mathscr{P}}$, a primitive group of degree
$p$ which by Lemma \ref{affineortwotransitive} is affine or almost simple and
2-transitive. If $Y^{\mathscr{P}} \cong Y$ is affine of degree $p$, then $Y
\leqslant  \AGL(1,p)$ and so $\rho(T) \triangleleft Y$ and the $\rho(T)$-orbits
are blocks of imprimitivity for $Y$ in $V\Gamma$, by Lemma \ref{blocks}(i),
whence the conclusion of Lemma \ref{casetwo} holds. Thus we may suppose that $Y$
is almost simple with socle $L$ and $Y^{\mathscr{P}}$ is 2-transitive with
$L^{\mathscr{P}}\cong L$ as in Table \ref{as2t}.\\\\
Since $Y^{\mathscr{P}}$ is 2-transitive, the quotient graph
$\Gamma_{\mathscr{P}} \cong K_p$. Let
$B\in\mathscr{P}$ and $\alpha\in B$. Now $L^{\mathscr{P}}$ is transitive,
and if $L$ is not transitive on $\VGamma$ then its orbits are blocks of
imprimitivity for $Y$ of size $p$ (by Lemma \ref{blocks}(i)) and as before the
conclusion of Lemma \ref{casetwo} holds. Thus we may assume
that $L$ is transitive on $\VGamma$, so $L_{\alpha} < L_{B} < L$, and $|L_B :
L_{\alpha}|=q$. Since $q\geq3$ and $q|(p-1)$, we have $p\geq 7$ and $q\leq
(p-1)/2$. We consider separately each line of Table \ref{as2t}. Note that, by
Lemma \ref{partition}, it is sufficient to prove either that $Y$ has a block of
imprimitivity of size $p$, or that $L_B$ has no subgroup of index $q$.\\\\
If $L=A_p$ with $p\geq 7$, then $L_B = A_{p-1}$ has no subgroup of index less
than $p-1$. If $L=\PSL(2,11)$ or $M_{11}$, with $p=11$, then $q=5$, so
$\Gamma=\Gamma(55,\ell,i)$, with $\ell=2$ or $5$ and $i=1$ or $2$. Using GAP we
construct each graph and verify that none has an almost simple automorphism
group. If $L=M_{23}$ then $q=11$ and $L_B=M_{22}$, which has no
subgroups of index 11 (see \cite[page 39]{conway1985atlas}).\\\\
Thus $L=\PSL(n,r)$, with $p=\frac{r^n-1}{r-1}$ and $n$ prime, and $r=r_0^f$ with
$r_0$ prime. First note that $n \geq 3$, for
if $n=2$ then $p=r+1$ and so $p-1=r$ is even and so is a power of $2$, and hence
not divisible by $q$ since $q\geq 3$.\\\\
Before seeking the subgroup $L_{\alpha}$ if index $q$ in $L_B$ we obtain some
further parameter restrictions. The subgroup $\rho(T)$, being cyclic of prime
order $p=\frac{r^n-1}{r-1}$, is a Singer cycle of $T$, is self-centralising, and
$N_Y(\rho(T)) \leqslant \rho(T).\Z{n}.\Z{f}$, so $|N_Y(\rho(T)) : \rho(T)|$
divides $nf$ (see \cite[Satz 7.3]{huppert1967endliche}). Since $\iota(H) \cong
\Z{\ell}$
normalises $\rho(T)$ it follows that $\ell$ divides $nf$ and that $\val\Gamma
\leq 2nf$. Moreover, since $\rho(T)$ is self-centralising, $T$ does not contain
$\lambda(X)$ and so, by Lemma \ref{L}, $q\nmid \ell$. Now the number of
$\Gamma$-edges with one vertex in $B$ is $|B|\val\Gamma \leq 2nfq$. On the other
hand since $\Gamma_{\mathscr{P}} = K_p$, this number is at least $p-1$, and
hence
\begin{equation}\label{valency}p-1\leq 2nfq.\end{equation}
Now $L_B = R\rtimes M \leqslant \AGL(n-1,r)$, where $R$ is elementary abelian of
order $r^{n-1}$, and $\SL(n-1,r) \leqslant M \leqslant \GL(n-1,r)$ with $M$ of
index $\gcd(n,r-1)$.
The group $L_B^B$ is transitive of prime degree $q$, and hence primitive.
Suppose first that $R^B \neq 1$. Since $R$ is a minimal normal subgroup of
$L_B$, $R$ acts faithfully and transitively on $B$, and since $R$ is abelian it
follows that $R^B$ is regular and $q=r^{n-1}$, forcing $n=2$ and a
contradiction. Thus $R^B=1$, and so $L_B^B=M^B$. Let $S=\SL(n-1,r)\leqslant M$.
If $S^B=1$ then $L_B^B$ is cyclic of order dividing $|M:S|$, which divides
$r-1$. Hence, by \eqref{valency}, $\frac{r(r^{n-1}-1)}{r-1} = p-1 \leq 2nf(r-1)
< 2nr(r-1)$ which implies $n=3$ (since $n$ is prime) and so $q$ divides $p-1 =
r(r+1)$. Since also $q$ divides $r-1$ it follows that $q=2$, a
contradiction.\\\\
Thus $S^B \neq 1$, so $S^B$ is primitive of odd prime degree $q$. Suppose first
that
$(n,r)=(3,2)$ or $(3,3)$, so $p$ is $7$ or $13$ respectively and $q=3$ is the
only odd prime dividing $p-1$. Since $q\nmid \ell$ we have only the following
two cases: $(p,q,\ell,i) = (13,3,2,1), (13,3,4,1)$ (since we are assuming that
$(p,q,\ell)\neq (7,3,2)$). It is easy to verify (say, in GAP) that the
automorphism groups of these graphs are as in Theorem \ref{autgamma}, and in
particular $Y$ has a block of imprimitivity of size $p$ so Lemma \ref{casetwo}
holds. Thus we may assume that $S$ is perfect and hence $S^B$ has $\PSL(n-1,r)$
as a compisition factor. In particular $S^B$ is an insoluble primitive group of
prime degree $q$ and so by Lemma \ref{affineortwotransitive}, $S^B \cong
\PSL(n-1,r)$ and either $q=\frac{r^{n-1}-1}{r-1}$, or
$(n,r,q)=(3,11,11),(3,5,5)$ or $(3,4,5)$. In the last case $p = 1+4+16 = 21$ is
not prime. In the previous two cases $Y = \PSL(3,r)$ does not contain a
Frobenius group $G_{pq}$. Thus $q=\frac{r^{n-1}-1}{r-1}$. Since $q$ is prime,
also $n-1$ is prime, and since $n$ is prime this implies $n=3$. Then $p=1+r+r^2$
and $q=1+r$. If $r=2$ we have the case excluded in Lemma \ref{casetwo}. If $r >
2$ then $q$ prime forces $r=2^a$ with $a$ even, which implies that $p=1+r+r^2$
is divisible by 3, a contradiction.
\qed\end{proof} 
Finally we consider the case $K=\lambda(X)$.
\begin{lemma}\label{kislambdax}
Assume \eqref{assumption} holds. If $K= \lambda(X)$ then the conclusion of Lemma
\ref{casetwo} holds.
\end{lemma}
\begin{proof}
Suppose $K=\lambda(X)$. Then $Y/K$ acts faithfully on the partition$
\mathscr{P}$, and so $Y/K$ is a transitive group of degree $p$, and so by Lemma
\ref{affineortwotransitive}, is either affine or 2-transitive and almost
simple.\\\\
If $Y/K$ is affine, then $Y^{\mathscr{P}}\leqslant \AGL(1,p)$, and so $\rho(T).K
\triangleleft Y$. Since $\rho(T)$ centralises $K=\lambda(X)$, $\rho(T)$ is a
characteristic subgroup of $\rho(T)K$ and hence $\rho(T) \triangleleft Y$. By
Lemma \ref{blocks}, the $\rho(T)$-orbits in $G$ are blocks of imprimitivity, and
the conclusion of Lemma \ref{casetwo} holds. Thus we may assume that $Y/K$ is
almost simple with socle as in Table \ref{as2t}.\\\\
Let $K < L \leqslant Y$ be such that $L/K = \Soc(Y/K)$. We consider the derived
group $L' \trianglelefteq L$. Since $K$ has prime order, either $K \subseteq L'$
or $K \cap L' =1$.\\\\
Case 1: $K \cap L' =1$:\\
In this case, $K$ and $L'$ are normal subgroups which intersect trivially, and
$L= L' \times K$. If $L'$ is intransitive then its orbits are blocks of size
$p$, and the conclusion of Lemma \ref{casetwo} holds by Lemma \ref{partition}.
So we may assume that $L'$ is transitive. The argument in the proof of Lemma
\ref{kis1} shows that $L'=\PSL(n,r)$ with $n$ an odd prime and $p =
\frac{r^n-1}{r-1}$. This time we have that $N_Y(\rho(T)) \leqslant (\lambda(X)
\times \rho(T)).\Z{n}.\Z{f}$. So here we have that $\ell$ divides $nfq$ (instead
of $nf$).\\\\
Since $Y^{\mathscr{P}}$ is 2-transitive, the quotient $\Gamma_{\mathscr{P}}
\cong K_p$. Moreover since $\mathscr{P}$ is the set of $\lambda(X)$-orbits there
is a constant $c$ such that each vertex in $B$ is joined to $c$ vertices in each
of the blocks distinct from $B$. Thus there are exactly $qc(p-1)$ edges of
$\Gamma$ with one vertex in $B$. On the other hand this number is $|B|\val\Gamma
= 2q\ell \leq 2q^2nf$, and so again the inequality \eqref{valency} holds: $p-1
\leq 2nfq$.\\\\
Now the rest of the argument in the proof of Lemma \ref{kis1} applies, ruling
out all parameter values except possibly $\PSL(3,r)$ with $q=3$ and $(r,p) =
(2,7)$ or $(3,13)$, for every $\ell$ dividing $p-1$ with $q\mid \ell$, and by
assumption, $\ell < p-1$. This leaves only the parameters $(p,q,\ell,i) =
(7,3,3,1), (13,3,3,1),(13,3,6,1)$. A computer check of these graphs confirms
that the conclusion of Lemma \ref{casetwo} holds in all cases.\\\\
Case 2: $K \subseteq L'$.\\
If $K \subseteq L'$ then $L$ is a perfect central extension of $L/K$, and so
(see \cite[Chapter 5.1]{gorenstein1998classification}), $K$ is a subgroup of the Schur
multiplier of $L/K$. Table \ref{as2t}
displays the Schur multipliers of the 2-transitive simple groups of prime
degree: since $q$ is an odd prime, we eliminate each case with a Schur
multiplier of size less than 3. We are left with only two possibilities: $A_7$
and $\PSL(n,r)$. In the former case we have $p=7$, implying that $q=\ell=3$. Then the only parameter sets possible are $(7,3,2,1), (7,3,3,1)$. The former yields the unique primitive example of Proposition \ref{primitive}, and the second is ruled out by computer search (as above in Case 1). In the latter case we have $\PSL(n,r)$, with $p=\frac{r^n-1}{r-1}$, in which
case the Schur multiplier is
cyclic of order $\gcd(r-1, n)$. Thus $q\mid r-1$ and $q\mid n$, and hence
$p=1+r+\dots
+r^{n-1} \equiv n\equiv 0 \pmod{q}$, but this implies $q \mid p$, which is a
contradiction.
\qed\end{proof}
The proof of Lemma \ref{casetwo} now follows from Lemmas \ref{K1},
\ref{KisDihedral}, \ref{kis1} and \ref{kislambdax}.
\subsection{Blocks of size $p$}\label{qblocks}
By Lemma \ref{casetwo}, the partition $\mathscr{P}=\{ Tg\mid g\in G\}$ is
$Y$-invariant. Since by \eqref{formula} $(z^{i})^{H} \subseteq z^i T$, the set
$S
\cap
z^j T$ has order $\ell$ or $0$, for any $j$. We dealt with the case $\ell=p-1$
in Lemma \ref{p-1}, and so we assume $\ell<{p-1}$. Recall that we also assume
$(p,q,\ell)\neq (7,3,2)$. 
\begin{lemma}\label{dihedral}
The quotient graph $\Gamma_{ T}$ is $K_2$ if $q=2$ and $C_q$ if $q$ is odd, and
$Y^{\mathscr{P}}$ is $\Z{2}$ or a subgroup of $D_{2q}$ containing $\Z{q}$
respectively.
\end{lemma}
\begin{proof}
If $q=2$ then $\Gamma_{\mathscr{P}} = K_2$ and $Y^{\mathscr{P}}\cong \Z{2}$, so
assume that $q$ is odd. Then $\Gamma_{\mathscr{P}}=\Cay(G/T, ST/T)$, and
$|ST/T|=|((z^{i})^{H}T)/T|+|((z^{-i})^{H}T)/T|$. Since $\iota(H)$ fixes the
cosets of $T$
setwise, we have $((z^{i})^{H}T)/T= \{z^iT\}$, and so $|ST/T|=2$. So since
$\Gamma_T$ is connected, it is a cycle.
\qed\end{proof}
\begin{lemma}\label{klem}
One of the following holds:
\begin{enumerate}
 \item The kernel $K= Y_{(\mathscr{P})}$ is $\rho(T).\iota(H)$ with $\rho(T)
\vartriangleleft Y$;
\item $(p,q,\ell,i) = (7,2,3,1), Y = \PGL(3,2).2$ and $\Gamma$ is the incidence
graph of $\PG(2,2)$;
\item $(p,q,\ell,i) = (11,2,5,1), Y = \PGL(2,11)$ and $\Gamma$ is the incidence
graph of the $(11,5,2)$-biplane; or
\item $(p,q,\ell,i) = (73,2,9,1), Y = \PGammaL(3,8).2$ and $\Gamma$ is the
incidence graph of $\PG(2,8)$.
\end{enumerate}
\end{lemma}
\begin{proof}
By Lemma \ref{action}(i), $\iota(H)$ fixes each coset of $T$ setwise, and so
$\iota(H)\leqslant K$. Also, since $T\triangleleft G$ it follows that $\rho(T)$
fixes each coset setwise, so $\rho(T)\leqslant K$. Thus it suffices to prove
that either $|K|\leqslant p\ell$, or one of cases (ii)-(iv) holds.\\\\
Claim: $K\cong K^T$.\\
Let $T'$ be a block adjacent to $T$ in $\Gamma_{\mathscr{P}}$. The pointwise
stabiliser $K_{(T)}$ is normal in $K$, and so $K_{(T)}^{T'} \triangleleft
K^{T'}$, which is primitive (being transitive of prime degree). By Lemma
\ref{blocks}, $K^{T'}$ is transitive or trivial. However $S \cap T'$ is fixed
setwise by $K_{(T)}$, and $|S \cap T'| \leq \ell < p$ so transitivity is
impossible. Thus $K_{(T)}$ fixes $T'$ pointwise. Since $\Gamma$ is
connected, we can repeat the same argument to show that $K_{(T)}$ acts trivially
on every block, and so $K_{(T)}=1$, and hence $K\cong K^T$.\\\\
Since $\rho(T)\leqslant K$, $K$ is transitive on each block of prime degree $p$,
so by Lemma \ref{affineortwotransitive}, this action is affine or 2-transitive
and almost
simple, and is given in Table \ref{as2t}. Assume first that the latter holds.
Now each almost simple 2-transitive group has at most 2 inequivalent actions
(see
\cite[Table 7.4]{cameron}), and so if $q\geq 3$ then there exist at least two
blocks on which $K$ acts equivalently, and by Corollary \ref{actionofK2} the
actions of $K$ on all blocks are equivalent.\\\\
If $q=2$ and the action of $K$ on the two blocks $T$ and $Tz$ are inequivalent,
then the only possibilities are $Y\leqslant \PGL(2,11)$ with $p=11$, or
$\PSL(n,r)\leqslant Y \leqslant \Aut(\PSL(n,r))$ with $p= \frac{r^n-1}{r-1}$ and
$n$ an odd prime. The former case can be
checked by a GAP calculation, or by hand, for both $\ell = 2,5$: the
graph $\Gamma(22,5,1)$ is the incidence graph of the $(11,5,2)$-biplane and $Y= \PGL(2,11)$, so part (iii) holds; and $K = \rho(T).\iota(H)$ holds for $\Gamma(22,2,1)\cong
C_{22}$.\\\\
In the latter case, $Y_{1}$ has orbits in $Tz$ of sizes $\frac{r^{n-1}-1}{r-1}$
and $r^{n-1}$ and hence $\ell$ is one of these integers. Since $\ell$ divides
$p-1$, we have $\ell = \frac{r^{n-1}-1}{r-1}$. However in this case a cycle of
length $p = \frac{r^n-1}{r-1}$ in $Y$ is a Singer cycle and the normaliser
$N_Y(\rho(T))$ has size $2pnf$, where $r=r_0^f$ with $r_o$ prime. So the
stabiliser $(N)Y(\rho(T)))_1$ has size $nf$. Since this subgroup contains
$\iota(H)$, it follows that $\ell$ divides $nf$. There are only two possible
choices of parameters $(r_0,f,n)$ satisfying this constraint along with the
constraint that $p=\frac{r^n-1}{r-1}$ is prime, namely $(r_0, ,f,n) =
(2,1,3),(2,3,3)$. These sets produce the exceptional graphs $\Gamma(14,3,1)$ and
$\Gamma(146,9,1)$, namely the incidence graphs of the Fano plane $\PG(2,2)$ and
of $\PG(2,8)$ respectively with $Y = \PGammaL(3,r).\Z{2}$  so that part (iii) or (iv) holds respectively. Assume now that none of parts (ii)-(iv) holds. Then (for all $q$)
the $K$-actions on all blocks are equivalent.
We now have that $K_1$ fixes a unique point $\alpha\in T'$. The set $T'\cap S$ of size $\ell$, where $1 <\ell < p-1$, is
fixed setwise by $K_1$ and so $K_1 = K_\alpha$ is not transitive on $T'
\setminus \{\alpha\}$. So $K^T$ is not 2-transitive, which is a
contradiction.\\\\
This completes consideration of the case where $K^T$ is insoluble. Suppose now
that $K^T \leqslant \AGL(1,p)$. Since $K^T$ is affine, all $K$-actions on blocks
are equivalent and the stabiliser in $K^T$ of two points is trivial. Thus $K_1$
fixes a point $\alpha\in T'$, and $T \cap \Gamma(\alpha)$ is fixed setwise by
$K_1$. Choose $\beta$ in this set: then the orbit-stabiliser theorem gives
$|\beta^{K_1}||K_{(1,\beta)}| = |K_1|$. But as $|\beta^{K_1}| \leq \ell$, we
have $|K_1| \leq \ell$, and so $|K| \leq p\ell$ as required.
\qed\end{proof}
\begin{proof}[Proof of Theorem \ref{autgamma}]
The four exceptional parameter sets are covered by Proposition \ref{primitive}
and Lemma \ref{klem}, so we may assume that $(p,q,\ell,i)\neq (7,3,2,1),
(7,2,3,1), \linebreak (11,2,5,1), (73,2,9,1)$. If $\ell = p-1$ the result follows from
Lemma \ref{p-1} so we may assume that $ 1 < \ell < p-1$. Then by Lemma
\ref{casetwo}, the cosets of $T$ form a $Y$-invariant partition $\mathscr{P}$ of
$V\Gamma$, and by Lemmas \ref{dihedral} and \ref{klem}, $Y^{\mathscr{P}} \cong
T/(\rho(T) \iota(H))$ is $\Z{q}$ or $D_{2q}$. Thus $\rho(G)\iota(H)$ has index
at most 2 in $Y$ and hence is normal in $Y$.\\\\
Suppose first that $q\nmid \ell$. Now $\rho(G)$ is characteristic in
$\rho(G).\iota(H)$, as it is the unique Hall $(p,q)$-subgroup (since neither $p$
nor $q$ divides $\ell$), and hence $\rho(G) \triangleleft Y$, so $Y$ is
contained in the holomorph $\Hol(G)=\rho(G).\Aut(G)$ of $G$. If $Y^{\Gamma_T}$
were dihedral there would be an automorphism that fixes $T$
and swaps the cosets $z^j T$ and $z^{-j}T$; but no such automorphism of $G$
exists (see Lemma \ref{autg}), and so $Y=\rho(G).\iota(H)$ in this case.\\\\
Now suppose that $q|\ell$. By Lemma \ref{L}(iii), $\Gamma$ is a normal
edge-transitive Cayley graph for
the abelian group $\rho(T)\times\lambda(X)$. The map $\sigma: x\mapsto x^{-1}$
is
an automorphism of $L$ since $L$ is abelian, and fixes $\Gamma(1)$ setwise. So
$\sigma$
is an automorphism of $\Gamma$, (in fact, it is in the normaliser
$N_{Y}(\rho(L))$), but is not contained in $\rho(G)\iota(H)$ as it swaps the
cosets $Tz^i$ and $Tz^{-i}$ and fixes the subgroup $T$. So $\rho(G)\iota(H)$ has
index 2 in $Y$ and so $Y=\rho(G)\iota(H).\Z{2}$.
\qed\end{proof}
\bibliographystyle{spmpsci}
\bibliography{references}
\end{document}